\documentclass[oneside]{amsart}
\usepackage{amsaddr}

\usepackage{amsthm}
\usepackage{float}
\usepackage{thmtools, thm-restate}

\usepackage{geometry}                		
\usepackage[utf8]{inputenc}
\usepackage[english]{babel}
\usepackage{tikz}

\theoremstyle{definition}

\newtheorem{thm}{Theorem}[section]
\newtheorem{defn}[thm]{Definition}
\newtheorem{cor}[thm]{Corollary}
\newtheorem{lem}[thm]{Lemma}
\newtheorem{rmk}[thm]{Remark}
\newtheorem{ex}[thm]{Example}

\theoremstyle{definition}

 \usepackage{amsmath,amssymb,amsthm}

\def\CC{{\mathbb C}}

\def\NN{{\mathbb N}}

\newcommand{\ds}{\displaystyle}

\def\<{\langle}
\def\>{\rangle}

\newcommand{\myTileslevel}[3]{%
	\foreach \X[count=\Z] in {#1}
	{\draw[thick,fill=\X] (#2*\Z,#3) -- (#2*\Z + #2,#3) -- (#2*\Z + #2,#3+#2) -- (#2*\Z,#3+#2) -- cycle;
	   \node at (#2*\Z + 0.5,-0.4) {\Z};
	}
}

\begin{document}
	
	\title{Tiling proofs of Jacobi triple product and Rogers-Ramanujan identities}

	\author{Alok Shukla}
	\address{Mathematical and Physical Sciences Division \\ School of Arts and Sciences \\ Ahmedabad University, India \\ alok.shukla@ahduni.edu.in, sajal.eee@gmail.com}
	\date{\today}
	\thanks{}

\begin{abstract}
	We use the method of tiling to give elementary combinatorial proofs of some celebrated $q$-series identities, such as Jacobi triple product identity, Rogers-Ramanujan identities, and some identities of Rogers. We give a tiling proof of the q-binomial theorem and a tiling interpretation of the q-binomial coefficients. A new generalized $ k $-product  $q$-series identity is also obtained by employing the `tiling-method', wherein the generating function of the set of all possible tilings of a rectangular board is computed in two different ways to obtain the desired  $q$-series identity. Several new recursive $ q$-series identities were also established. The `tiling-method' holds promise for giving an aesthetically pleasing approach to prove old and new $q$-series identities.
	\end{abstract}	
	
	\maketitle
		
\section{Introduction}
	
	Jacobi triple product identity and Rogers-Ramanujan identities are among the most famous $q$-series identities. Several proofs and generalizations, including Lie algebraic interpretations of these identities, are available in the literature. Besides their importance in number theory, surprisingly, Rogers-Ramanujan identities also appear in the Hard Hexagon model in statistical mechanics \cite{baxter1981rogers}.  
	
	Recently, a beautiful `tiling proof' of Euler's Pentagonal Number Theorem by Eichhorn, Nam, and Sohn appeared in \cite{eichhorn2019tiling}. The proof in \cite{eichhorn2019tiling} was inspired by works of Little and Sellers in \cite{MR2469260}  and \cite{MR2587022}. Little and Sellers used a new combinatorial method for proving partition identities. They used tilings of a $1 \times \infty$ board with squares and dominoes.
	The main idea in \cite{eichhorn2019tiling}, \cite{MR2469260}  and \cite{MR2587022} is to compute the generating function of the set of all tilings of a board in two different ways to establish a $q$-series identity.
	These tiling proofs are quite elementary and provide an appealing method to revisit some old classical $ q $-series identities.
	
	The authors stated in the conclusion section of their paper \cite{eichhorn2019tiling} that it remains an open problem whether or not Jacobi's triple product identity can be obtained using the method of tiling. 	We will use the method of tiling to prove Jacobi triple product identity. Moreover, we will also prove Rogers-Ramanujan identities using a similar method. In the course of proving these results, we will also prove the $ q $-binomial theorem considering tilings of a $ 1 \times N $ board with black and white tiles. A tiling interpretation of the $ q $-binomial coefficients will also be given.  The $ q $-binomial theorem  will be employed, along with an idea due to Cauchy, to prove Jacobi triple product identity. The idea of Cauchy \cite{cauchy1893second} was to use a finite version of the Jacobi triple product identity from which the main result is obtained by taking an appropriate limit. Our proof of Rogers-Ramanujan identities will also follow the same approach and it will closely follow \cite{bressoud1983easy}.  
	
	We acknowledge that for some of the results presented here, elementary partition theoretic proofs are known, and for some others, such as the identities of Rogers proved in Section \ref{Sec_Rogers}, tiling proofs are also known. Still, a unified approach will be presented here to prove these results using the method of tiling. Besides proving some old classic results, we will also prove new results.
	
	The concept of the `level of a tiling' will be introduced in Section \ref{Sec_Ramanujan}. The following result will be proved in Theorem \ref{thm_rogers_ramanujan_main} using the `level of a tiling'. 
      
      	\begin{equation*}
      		\prod_{j=1}^{N} \frac{1}{1 - z_j} = 1 + \sum_{j=1}^{N} \, \,  \frac{h_j(z_j, z_{j+1}, \cdots, z_N) }{(1-z_1) (1-z_2) \ldots (1-z_j)},
      	\end{equation*}
      	where, 
      	\begin{equation*}
      		h_j(z_j, z_{j+1}, \cdots, z_N)	:= \sum_{ \substack{
      				0 \leq \,  \beta_j, \beta_{j+1}, \cdots, \beta_N \leq \,  j \\  \sum_{k=j}^{N} \beta_k = j
      		}}   \quad  \,  {\prod_{k=j}^{N} z_k^{\beta_k} },
      	\end{equation*}
      	and $ z_j \in \CC $ with  $ |z_j| < 1  $ for $ 1 \leq j \leq N $.

	We will also prove the following new $ k $-product identity (Corollary~\ref{cor_main_recur_k}, Theorem~\ref{thm_main_new_recur_k}) using the method of tiling. 
	\begin{align*}
	\prod_{m=1}^{\infty} (1 + z_1 q^m) (1 + z_2 q^m) \cdots  (1 + z_k q^m) &= \sum_{m=0}^{\infty} 	\frac{q^{m} }{(1-q)(1-q^2) \cdots (1-q^m)} P(m),
	\end{align*}
	where, $ P(m) $ satisfies the following recurrence 
	\begin{equation*}
	P(m) = \sum_{i=1}^{k} \frac{q^{m-i} (q)_{m-1}}{(q)_{m-i}} \, S_i^k (z_1, z_2, \cdots, z_k) P(m-i).
	\end{equation*}
	Here, the notation $ (q)_m $ is defined in \eqref{eq_q_n}, and  $S_n^k (z_1, z_2, \cdots, z_k) $ is the elementary symmetric polynomial of degree $ n $, in variables $ z_1, z_2, \cdots, z_k $, precisely defined in \eqref{eq_elem_symmetric_poly}.
	We note that $ P(m) $ can alternately be written as $ P(m) = q^{-m} (q)_m F(m)  $, with $ F(m) $ given in a closed form in \eqref{eq_new_F_m} in the statement of Theorem~\ref{thm_main_new_recur_k_alternate}.

	All the proofs presented in this work are elementary and use combinatorial arguments involving tilings of different types of boards. The only analytical result needed will be Tannery’s theorem (see Sect.~3.7.1, \cite{loya2017amazing}), which is just a special case of the Weierstrass M-test, or the Lebesgue dominated convergence theorem.

\section{Basic setup and preliminary examples} \label{sec_basic_setup}

In the following we will use the usual $ q $-series notation 
\begin{equation}\label{eq_def_q_notation}
(z;q)_n := \prod_{m=0}^{n-1} (1 - zq^m), \quad {\text{and}} \quad (z;q)_{\infty} := \lim\limits_{n \to \infty} (z;q)_n.
\end{equation}
Here, and in the rest of the paper, $ z $ and $ q $ will be complex numbers with $ |q| < 1 $ and $ z \neq 0 $.
To ease the notation we will use 
\begin{equation}\label{eq_q_n}
(q)_n := (q;q)_n = \prod_{m=1}^{n} (1 - q^m), \quad {\text{and}} \quad (q)_{\infty} := \lim\limits_{n \to \infty} (q)_n.
\end{equation}
The \emph{Gaussian polynomial} or $ q $-binomial coefficient $ \left[\substack{N \\ \\ m } \right] $ will be defined as
\begin{equation}\label{eq_Gaussian_polynomial}
\left[\substack{N \\ \\ m } \right] := \begin{cases}
&\frac{(q)_N}{(q)_m (q)_{N-m}}  \quad \text{if $ 0 \leq m \leq N $,}\\
& 0 \quad \text{otherwise.}
\end{cases}
\end{equation}
 
We will consider the tilings of a $ 1 \times N $ board, i.e., a board with $ 1 $ row and $ N $ columns, with white and black square tiles. A tiling $ T $ of $ 1 \times N $ board consists of $ N $ square tiles covering the board, such that each of the $ N $ squares in a tiling could be either a white square or a black square, with exactly one square tile at each position. If $ N $ is $ \infty $, then we further require that \emph{each tiling must consist of only a finite number of black squares}. More formally, we have the following definition.
\begin{defn}
	A tiling  $ T $ of a $ 1 \times N $ board is a function $ T : \{1,2,3, \cdots, N\} \to \{white,black\} $. So, essentially a tiling assigns either a white square or a black square for each position on the board. For a $ 1 \times \infty $ board, a tiling $ T $ must satisfy the additional condition that $ T(i) = black $ for only finitely many positive integers $ i $.
\end{defn}

For a $ 1 \times N $ board $ F $, the associated weight function $ F_w $ assigns a fixed weight $ F_w(i) \in \CC[z,q] $ for each position $ i$ (see Fig.\!~\ref{fig_wieght_board}  for an example.)

\begin{figure}[ht]
	\begin{center}
		\begin{tikzpicture}[scale = 1.4]	
		\foreach \Z in {1,2,3}
		{
			\draw[thick] (\Z,0) -- (\Z + 1,0) -- (\Z + 1,1) -- (\Z,1) -- cycle;
			\node at (\Z + 0.5,0.5) {$ zq^{\Z} $};
			\node at (\Z + 0.5,-0.4) {$ \Z$};
		}
		\foreach \x in {4,5,6}
		\draw[thick,dashed] (\x,0) -- (\x + 1,0) -- (\x + 1,1) -- (\x,1) -- cycle;
		
		\foreach \x in {7,8}
		\draw[thick] (\x,0) -- (\x + 1,0) -- (\x + 1,1) -- (\x,1) -- cycle;
		
		\node at (5 + 0.5,0.5) {$ zq^{i} $};
		\node at (5 + 0.5,-0.4) {$ i$};
		\node at (7 + 0.5,0.5) {$ zq^{N-1} $};
		\node at (7 + 0.5,-0.4) {$ N-1$};
		\node at (8 + 0.5,0.5) {$ zq^{N} $};
		\node at (8 + 0.5,-0.4) {$ N$};
		\end{tikzpicture}
	\end{center}
	\caption{A $ 1 \times N $ board with its associated weight function $ F_w(i) = zq^i$.}\label{fig_wieght_board}
\end{figure}
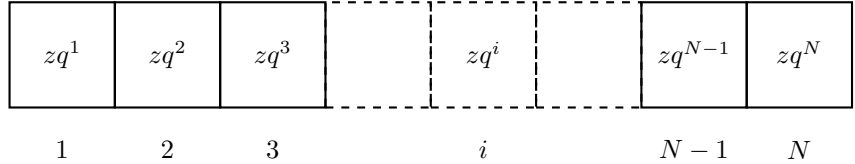

\begin{defn} \label{def_local_weight}
	For a tiling $ T $ of $ F $, its local weight $ W_T(i) $ at the position $ i $ is defined as 
	\begin{equation}\label{eq_WT_i}
	W_T(i) :=  \begin{cases} F_w(i)  & \text{if the tile at  $ i $-th position is black, i.e., if $T(i) = black$,} \\
	1 & \text{if the tile at $ i $-th position is white, i.e., if $T(i) = white$.} \end{cases}
	\end{equation}

The total weight of a tiling $T$ is defined as 
\begin{equation}\label{eq_weight}
W(T) := \prod_{i=1}^{N} W_T(i),
\end{equation}
the product of its local weights at all the positions on the board. This definition can naturally be extended to a $ 1 \times  \infty $ board with an infinite product of local weights. Since in any tiling $ T $ only finitely many tiles are black, and white tiles contribute $ 1 $, the infinite product of local weights  makes sense. Also, since effectively only black tiles contribute to the total weight, in the following we will only consider the placement of the black tiles to construct a tiling.
\end{defn}

\begin{ex} \label{example1}
	Consider a $ 1 \times 10 $ board $ B $ with the associated weight function $ F_w(i) = zq^i $. Let $ T $ be the tiling of $ F $ with $ 3 $ black squares, each at positions $2,4$, and $7$, see Fig.\!~\ref{fig_example}.
	Since the white tiles at the remaining positions have weight $1$, they do not contribute to the total weight of the tiling $ T $. 
	Therefore, the total weight of the tiling $ T $ is 
	$ \displaystyle
W(T) = zq^2 \times zq^4 \times zq^7 = z^3 q^{13}.
	$
\end{ex}

\begin{figure}
	\begin{center}
		\begin{tikzpicture}[scale=1.3]
		\foreach \i in {1,2,3,...,10}
		\node at (\i+0.5,1.3) {$ zq^{\i} $};
		\myTileslevel{white,black,white,black,white,white,black,white,white,white}{1}{0};
	\end{tikzpicture}
	\end{center}
\caption{The total weight of the tiling $ T $ of the $ 1 \times 10 $ board shown in the figure is $ zq^2 \times zq^4 \times zq^7 = z^3 q^{13}  $. Note that only the positions with black tiles contribute in the total weight.}\label{fig_example}
\end{figure}
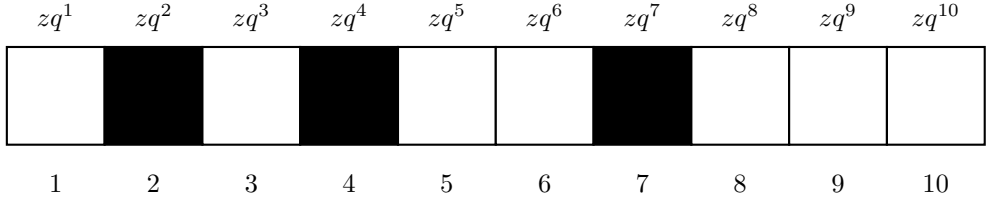

\begin{defn}\label{def_generating_fun}
	We define the \emph{generating function} of all tilings of a $1 \times N$  board $ F $ as
	$$
	H(F, z, q) := \sum_{\text{tilings} \ T \ \text{of F}} W(T).
	$$
\end{defn}
Suppose, the board $ F $ has the associated weight function $ F_w(i) = zq^i $. We note that the contribution to the total weight of the tiling of the tile at position $i$ must be either $1$ or $zq^i$, therefore
\begin{equation}\label{simpleF}
H(F, z, q) = \prod_{i=1}^{N} (1+zq^i) = (-zq;q)_{N} = (-zq)_N.
\end{equation}
We obtain, \begin{equation}\label{eq_gen_fun_finite}
	(-zq;q)_{N} = \sum_{\text{tilings} \ T \ \text{of F}} W(T).
\end{equation}
This simple relation will play an important role in the rest of this paper. Also, clearly for a $ 1 \times \infty $ board $ F $ the following holds
  \begin{equation}\label{eq_gen_fun_infinite}
 (-zq;q)_{\infty} = \sum_{\text{tilings} \ T \ \text{of F}} W(T).
 \end{equation}
 
\begin{defn} \label{def_bmk_bm}
	For a $ 1 \times N$ board $ F $, with $ 1 \leq N \leq \infty$, we define
	\begin{equation}{\label{eq_b_k_m}}
	F(k,m) :=	\sum_{\substack{T \ \text{is a tiling of} \ F \ \\ \text{with exactly} \ m \ \text{black tiles} } }  \prod_{i=k}^{N} W_T(i),
	\end{equation}
	where $W_T(i)$ is the same as defined in \eqref{eq_WT_i}. One can think of $ F(k,m) $ as the generating function for all the possible tilings of the board $ F $ using $ m $ black tiles starting at the position $ k $, i.e., from the position $ 1 $ to $ k-1 $ no black tiles are allowed to be placed. 
	We define 
	\begin{equation}\label{eq_b_m}
	F(m) : = F(1,m)
	\end{equation}
	for easing the notation. Clearly, $ F(m) $ is the generating function for all the tilings of $ F $, using $ m $ black tiles, starting at the position $ k=1 $.
	
\end{defn}

\subsection{An identity of Euler}
Consider a $ 1 \times \infty $ board $ F $ with the associated weight function $ F_w(i) = zq^i $. Clearly, 
the generating function of all the possible tilings of $ F $ is given by
\begin{equation}\label{simpleFone}
H(F, z, q) = \prod_{i=1}^{\infty} (1+zq^i) = (-zq;q)_{\infty} = (-zq)_{\infty}.
\end{equation}
Now we find this generating function in another way by first computing $ F(m) = F(1,m) $, the generating function for all the tilings of $ F $, using $ m $ black tiles. For this, we fix the first tile to be black with contribution $ zq $ and the remaining $ m-1 $ black tiles are placed with the corresponding generating function $ F(2,m-1) $. The total contribution to the generating function in this case is $ zq F(2,m-1)  $.  Next, we consider the case when the first tile is white. The total contribution in this case is $ F(2,m) $.
Therefore, we have
\begin{align}\label{eq_euler_recur_one}
 F(1,m) = zq F(2,m-1) + F(2,m).
\end{align}
If a black tile is shifted  one position to the right of its current position, then it will gain weight by a factor of  $q $. Therefore, if a tiling $ T $,  containing $ m $ black tiles, is shifted one position to its right, without changing the relative positions of the black tiles in it, then the weight of $ T $ will change by a factor of $ q^{m} $. Hence, the effect of multiplication by $ q^{m} $ on $ F(1,m) $ is to give the generating function of $ F(2,m) $. Therefore, we have
\begin{equation}\label{eq_euler_recur_two}
F(2,m) = q^m F(1,m),
\end{equation}
and, similarly 
\begin{equation}\label{eq_euler_recur_general}
	F(k+1,m) = q^{km} F(1,m).
\end{equation}
From \eqref{eq_euler_recur_one} and \eqref{eq_euler_recur_two} it follows that
\begin{align} \label{eq_euler_recur_relation}
F(1,m)&= zq^m F(1,m-1) + q^m F(1,m) \nonumber\\
& \implies F(1,m) =  \frac{zq^m}{1-q^m} F(1,m-1).
\end{align}
Since, $ F(1,1) = zq + zq^2 + zq^3 + \cdots = zq(1-q)^{-1} $, we get 
\begin{equation} \label{eq_euler_identity_m}
F(1,m) =  \frac{z^m q^{\frac{m(m+1)}{2}} }{(1-q)(1-q^2)\cdots (1-q^m)},
\end{equation}
and \begin{equation}\label{eq_Euler}
H(F, z, q) = \prod_{i=1}^{\infty} (1+zq^i)  = \sum_{m=0}^{\infty}  \frac{z^m q^{\frac{m(m+1)}{2}} }{(1-q)(1-q^2)\cdots (1-q^m)},
\end{equation}
which proves an identity of Euler using the tiling method.  This identity is also proven in \cite{eichhorn2019tiling} using the method of tiling, however with a slightly different approach. 

\begin{rmk}
	We give another combinatorial  interpretation of \eqref{eq_euler_recur_relation}. 
For this, we claim that
	\begin{align}\label{eq_remark_combinatorial_meaning}
		F(1,m) = zq F(2,m-1) + z q^2 F(3,m-1) + z q^3 F(4, m-1) + \cdots .
	\end{align}
Here $ F(1,m) $ is the generating function for all the tilings of the board $ F $ with  exactly $ m $ black tiles. We note that the first black tile in any tiling can occur at any of the positions $ 1, 2, 3 , \ldots $. It is clear that the $ k $-th term in the sum on the right side, i.e.,  $ z q^k F(k+1,m-1) $, is the generating function of all the tilings of the board $ F $ such that the first black tile occurs at the $ k$-th position on the board. This proves our claim in \eqref{eq_remark_combinatorial_meaning}. Using  $  F(k+1,m-1) = q^{k(m-1)} F(1,m-1) $ from \eqref{eq_euler_recur_general}  in \eqref{eq_remark_combinatorial_meaning}, for $ k=1,2,3,\ldots $, factoring out $ F(1,m-1) $ and computing the resulting geometric series yields $ \eqref{eq_euler_recur_relation} $.
\end{rmk} 
 
\section{Jacobi triple product identity}
	\begin{figure}[ht]
	\begin{center}
		\begin{tikzpicture} [scale =1.4]	
		\foreach \Z in {1,2,3}
		{   \pgfmathtruncatemacro{\i}{\Z-1};
			\draw[thick] (\Z,1) -- (\Z + 1,1) -- (\Z + 1,2) -- (\Z,2) -- cycle;
			\node at (\Z + 0.5,1.5) {$ z^{-1}q^{\i} $};
		}
		\foreach \x in {4,5,6}
		\draw[thick,dashed] (\x,1) -- (\x + 1,1) -- (\x + 1,2) -- (\x,2) -- cycle;
		
		\foreach \x in {7,8}
		\draw[thick] (\x,1) -- (\x + 1,1) -- (\x + 1,2) -- (\x,2) -- cycle;
		
		\node at (5 + 0.5,1.5) {$ z^{-1}q^{i-1} $};
		\node at (7 + 0.5,1.5) {$ {\footnotesize z^{-1}q^{N-2} }$};
		\node at (8 + 0.5,1.5) {${\footnotesize  z^{-1}q^{N-1}} $};

		\foreach \Z in {1,2,3}
		{
			\draw[thick] (\Z,0) -- (\Z + 1,0) -- (\Z + 1,1) -- (\Z,1) -- cycle;
			\node at (\Z + 0.5,0.5) {$ zq^{\Z} $};
			\node at (\Z + 0.5,-0.4) {$ \Z$};
		}
		\foreach \x in {4,5,6}
		\draw[thick,dashed] (\x,0) -- (\x + 1,0) -- (\x + 1,1) -- (\x,1) -- cycle;
		
		\foreach \x in {7,8}
		\draw[thick] (\x,0) -- (\x + 1,0) -- (\x + 1,1) -- (\x,1) -- cycle;
		
		\node at (5 + 0.5,0.5) {$ zq^{i} $};
		\node at (5 + 0.5,-0.4) {$ i$};
		\node at (7 + 0.5,0.5) {$ {\footnotesize zq^{N-1}} $};
		\node at (7 + 0.5,-0.4) {$ N-1$};
		\node at (8 + 0.5,0.5) {$ zq^{N} $};
		\node at (8 + 0.5,-0.4) {$ N$};
		
		\node at (0.4,0.4) {$ G$};
		\node at (0.4,1.4) {$ F$};
		
		\end{tikzpicture}
	\end{center}
	\caption{A $ 1 \times N $ board $ F $ placed on the top of another $ 1 \times N $ board $ G $.  The weight functions of $ F $ and $ G $ are given by $ F_w(i) = z^{-1}q^{i-1} $  and $ G_w(i) = zq^i$ respectively}\label{fig_wieght_boards_AB}
\end{figure}
In this section, we will prove Jacobi's triple product identity. But, first, we give a combinatorial `tiling-proof' of the $ q $-binomial theorem.
  \begin{thm} \label{thm_bm}[q-Binomial Theorem]
  	Let $ F $ be a $ 1 \times N $ board with the associated weight function $ F_w(i) = zq^i $. Suppose $ F $ is tiled using $ m $ black tiles with $ 0 \leq m \leq N $. Then \begin{equation}\label{eq_A_m}
  F(m) :=	\sum_{\substack{T \ \text{is a tiling of} \ F \ \\ \text{with exactly} \ m \ \text{black tiles} } } W(T) = z^m q^{\frac{m(m+1)}{2}} 	\left[\substack{N \\ \\  m} \right],
  	\end{equation}
  	where $ W(T) $ is the total weight of the tiling $ T $ as defined in \eqref{eq_weight}. Therefore,
  \begin{equation}\label{key}
  	\prod_{m=1}^{N}\, (1 + z q^m) = \sum_{m=0}^{N} z^m q^{\frac{m(m+1)}{2}} 	\left[\substack{N \\ \\  m} \right].
  \end{equation}
  \end{thm}
\begin{proof}
The result given in \eqref{eq_A_m} is clearly true for $ m =N $ and $ m =0 $, so assume $ 1 \leq m \leq N-1 $.
Recall that $ F(k,m) $ denotes the generating function for all the possible tilings of the board $ F $ using $ m $ black tiles starting at the position $ k $, c.f.~\eqref{eq_b_k_m}. We observe that $ F(m) = F(1,m) $ consists of two types of tilings, with either a black tile or a white tile at the first position. If the first tile is black, then the first position will contribute $ zq $ and the remaining $ m-1 $ tiles will contribute $ F(2,m-1) $ to the generating function, yielding the total contribution $ zq F(2,m-1) $. If the first tile is white, then clearly the $ m $ black tiles will contribute $F(2,m) $ to the generating function. It means we have the following recurrence relation.
\begin{equation}\label{eq_recur_one_B_m}
F(1,m) = zq F(2,m-1) + F(2,m).
\end{equation}
Next, imagine that the original board is extended with an extra position appended to the right. Assume this extra $ (N+1) $-th position has the weight $zq^{n+1}  $. Let us call this new board $ \tilde{F} $. So, $ \tilde{F} $ is a $ 1 \times (N+1)$ board with the associated weight function $ \tilde{F}_w(i) = zq^i $. Let $ \tilde{F}(k,m) $ be defined analogous to \eqref{eq_b_k_m}, i.e., replacing $ F $ by $ \tilde{F} $ and $ N $ by $ N+1 $ in \eqref{eq_b_k_m}.

 We note that if any tiling $ T $ of the board $ F $ with $ m $ black tile is shifted right by one position, then we get a tiling $ \tilde{T} $ of the board $ \tilde{F} $ starting at the second position. Since shifting a black tile by one position by right results in its weight being multiplied by $ q $, and there are $ m $ black tiles in $ T $, it is obvious that $ q^m W(T) = W(\tilde{T}) $. Therefore, $ q^m F(1,m)  = \tilde{F}(2,m)$. The generating function of all the tilings of $ \tilde{F} $, with $ m $ black tile, such that a white tile is at the right most (i.e., $ (N+1) $-th) position, is $ F(2,m) $. Whereas, the generating function of all the tilings of $ \tilde{F} $, with $ m $ black tiles, such that a black tile is at the $ (N+1) $-th position, is $ qz^{N+1} F(2,m-1) $. Therefore, we obtain
\begin{equation}\label{eq_recur_two_B_m}
q^m F(1,m) = F(2,m) + zq^{N+1} F(2,m-1).
\end{equation}
From \eqref{eq_recur_one_B_m} and \eqref{eq_recur_two_B_m} it follows that
\begin{equation}\label{eq_Bm_recur_formula}
F(1,m) = z q \frac{(1-q^{N})}{(1-q^m)} F(2,m-1).
\end{equation}
One can use \eqref{eq_Bm_recur_formula}, replacing $ m $ by $ m-1 $, $ N $ by $ N-1 $, and $ zq $ by $ zq^2 $, to express $ F(2,m-1) $ in term of $ F(3,m-2) $.
A repeated application of \eqref{eq_Bm_recur_formula} and using fact that $ F(m,0) = 1 $ gives
\begin{equation}\label{eq_Bm_recur_formula_final}
F(1,m) =  \prod_{i=0}^{m-1} z q^{i+1} \frac{(1-q^{N -i})}{(1-q^{m-i})} = z^m q^{\frac{m(m+1)}{2}}  \left[\substack{N \\ \\  m} \right],
\end{equation}
and the proof is complete.
\end{proof}
We obtain the following corollary of the above theorem.
\begin{cor} \label{Cor_partition_interpretation}
	\begin{equation}\label{Eq_cor_main}
		 \left[\substack{N \\ \\  m} \right] = \sum_{ \substack{ \alpha_{1} + \alpha_{2} + \cdots + \alpha_{N-m} \leq m, \\ \alpha_{1}, \alpha_{2}, \cdots ,\alpha_{N-m} \in \{0,1,2, \ldots, m\}}} \, q^{\alpha_1 + 2 \alpha_{2} + 3 \alpha_{3} + \cdots + (N-m) \alpha_{N-m}}.
	\end{equation}
\end{cor}
\begin{proof}
	Consider the $ 1 \times N $ board $ F $ used in Theorem \ref{thm_bm}. Suppose $ F $ is tiled using $ m $ black tiles with $ 0 \leq m \leq N $. Let $ F(m) $ be as in 
	\eqref{eq_A_m}. Next, we compute $ F(m) $ in a different way. For this, initially we consider the tiling in which all the $ m $ black tiles are placed on the left of the board, with the black tiles occupying the first position through the $ m $-th position  on the board. For convenience, we call this tiling $ I $. Clearly, the total weight of the tiling $ I $, is given by $ W(I)  = z^m q^{\frac{m(m+1)}{2}} $.  Now we describe how any given tiling, say $ J $, can be transformed to the tiling $ I $ by moving the black tiles in $ J $ by using the following scheme. Suppose the board $ F $ is tiled with the tiling $ J $, and a person traverse the board from the right to the left. The first black tile encountered by the person is marked $ m $, the second black tile encountered is marked $ m-1 $ and so on. The $ i $-th black tile met by the person is marked $ m-i+1 $. Finally, the last black tile (the $ m $-th tile) met is  marked $ 1 $. Next, all the black tiles are removed from the board and then re-positioned according to their markings, with the black tile marked $ i $ getting placed at the $ i $-th position on the board, for $ i=1,2,3,\ldots, m $. This transforms the tiling $ J $ to the tiling $ I $. 
	We note that in this scheme any black tile in $ J $ is moved to the left by either $ 0 $, or  $ 1 $, or $ 2 $, $ \cdots $, or $ N-m $ positions. Suppose $ \alpha_{i} $ is the number of tiles that are moved to the left by $ i $ positions. Clearly, $ 0 \leq \alpha_{i} $.   Moreover, $ \alpha_{i} =0  $ if $ i > N-m $, as clearly the maximum possible left movement of a black tile is by $ N-m $ positions.  Also, as there are total $ m $ black tiles we must have $ \sum_{i=0}^{N-m} \alpha_{i} = m $. Since in a tiling moving one black tile to the left by $ i $ positions affects the total weight function by a factor of  $ q^{-i} $, moving $ \alpha_{i} $ black tiles in a tiling by $ i $ positions to the left will affect the total weight function by a factor of $ q^{-i \alpha_{i}} $. It is clear that 
	\begin{align*}
	 	W(I) &= W(J) q^{-\left(\alpha_1 + 2 \alpha_{2} + 3 \alpha_{3} + \cdots + (N-m) \alpha_{N-m}\right)}  \\
		 \implies W(J) &= z^m q^{\frac{m(m+1)}{2}} q^{\alpha_1 + 2 \alpha_{2} + 3 \alpha_{3} + \cdots + (N-m) \alpha_{N-m}},  
	\end{align*}
	where we used the fact that  $ W(I) = z^m q^{\frac{m(m+1)}{2}} $. It is now obvious that
	\begin{align} \label{Eq_cor_fm}
		F(m) =  z^m q^{\frac{m(m+1)}{2}} \sum_{ \substack{\alpha_0 + \alpha_{1} + \alpha_{2} + \cdots + \alpha_{N-m} = m, \\ \alpha_0, \alpha_{1}, \alpha_{2}, \cdots ,\alpha_{N-m} \in \{0,1,2, \ldots, m\}}} \, q^{\alpha_1 + 2 \alpha_{2} + 3 \alpha_{3} + \cdots + (N-m) \alpha_{N-m}}.
	\end{align}
Form \eqref{eq_A_m} and  \eqref{Eq_cor_fm} we obtain
\begin{equation*}
	\left[\substack{N \\ \\  m} \right] = \sum_{ \substack{\alpha_0 + \alpha_{1} + \alpha_{2} + \cdots + \alpha_{N-m} = m, \\ \alpha_0, \alpha_{1}, \alpha_{2}, \cdots ,\alpha_{N-m} \in \{0,1,2, \ldots, m\}}} \, q^{\alpha_1 + 2 \alpha_{2} + 3 \alpha_{3} + \cdots + (N-m) \alpha_{N-m}}, 
\end{equation*}
from which \eqref{Eq_cor_main} follows.
\end{proof}
\begin{rmk}
	Corollary \ref{Cor_partition_interpretation}  gives a tiling proof of the well-known combinatorial fact that  $  \left[\substack{n \\ \\  m} \right]  $ is the generating function of all partitions with at most $ m $ parts, such that the largest part is at most $ n-m $. 	
\end{rmk}

The following $ k $-product formula easily follows from the $ q $-binomial theorem.
\begin{cor}
	\begin{align}\label{eq_thm_q_multinomial}
		&	\prod_{m=1}^{N} (1 + z_1 q^m) (1 + z_2 q^m) \cdots (1+ z_k q^m)  \nonumber \\ 
		& 	\hspace{3cm} = \sum_{m_1, m_2, \cdots, m_k} z_1^{m_1} q^{\frac{m_1^2+m_1}{2}} \left[\substack{N \\ \\ m_1 } \right]  z_2^{m_2} q^{\frac{m_2^2+m_2}{2}} \left[\substack{N \\ \\ m_2 } \right] \cdots z_k^{m_k} q^{\frac{m_k^2+m_k}{2}} \left[\substack{N \\ \\ m_k } \right],
	\end{align}
where the summation is over  $ m_1, m_2, \cdots,m_k $, varying from $ 0$ to $ N $.
\end{cor}
\begin{proof}
	Let $ F_1$, $ F_2 , \cdots, F_k $, be $ 1 \times N $ boards.  Let the associated weight function for the board $ F_j $ be $ F_{j_w}(i) =  z_j q^i $.
	We can consider all these boards as a single combined board $ F $. 
	The left side of \eqref{eq_thm_q_multinomial} is the generating function of all the possible tilings of the combined board $ F $. On the other hand, any tiling of the board $ F $ will be such that $ m_j $ black tiles are on the board $ F_j $, contributing  $z_j^{m_j} q^{\frac{m_j^2+m_j}{2}} \left[\substack{N \\ \\ m_j } \right]  $ to the generating function. Therefore, the right side of \eqref{eq_thm_q_multinomial} is also the generating function of all the possible tilings of the combined board $ F $.
\end{proof}
Next, we prove the following lemma.
\begin{lem} \label{lem_lem_qbinomial} 
	\begin{equation}\label{eq_lem_qbinomial}
	\left[\substack{2N \\ \\  t } \right] = \sum_{i +j = t} q^{i(N-j)} 	\left[\substack{N \\ \\  i } \right]  	\left[\substack{N \\ \\  j} \right] 
	= \sum_{i =0}^{t} q^{i(N-(t-i))} 	\left[\substack{N \\ \\  i } \right]  	\left[\substack{N \\ \\  t-i} \right].
	\end{equation}
\end{lem}
\begin{proof}
	Consider a $ 1 \times 2N $ board $ F $ with the associated weight function $ F_w(i) = q^i$. Let $ T $ be a tiling of $ F $ with $ t $ black tiles. Imagine breaking the board $ F $ after $ N $ tiles to get two boards, say the left board $ L $ and the right board $ R $, each of size $ 1 \times N $. Then $ T $ gives a tiling of the left board with $ i $ black tiles and the right board with $ j $ black tiles such that $ i + j = t $. Conversely, each tiling of the left board with $ i $ black tiles and the right board with $ j $ black tiles, with $ i+j =t $, gives a tiling of the board $ F $ with $ t $ black tiles.
	Therefore, using \eqref{eq_Bm_recur_formula_final} we obtain
	\begin{align*}
	F(t) &= \sum_{i+j =t} \, L(i) R(j) \\
	\implies  q^{\frac{t(t+1)}{2}} \left[\substack{2N \\ \\  t } \right] &= \sum_{i+j =t} q^{\left(Nj +\frac{i(i+1) + j(j+1)}{2}  \right)} \left[\substack{N \\ \\  i } \right]  	\left[\substack{N \\ \\  j} \right] \\
	\implies  \left[\substack{2N \\ \\  t } \right] &= \sum_{i+j =t} q^{i(N-j) } \left[\substack{N \\ \\  i } \right]  	\left[\substack{N \\ \\  j} \right] 
	=  \sum_{i =0}^{t} q^{i(N-(t-i))} 	\left[\substack{N \\ \\  i } \right]  	\left[\substack{N \\ \\  t-i} \right].
	\end{align*}	
\end{proof}
\begin{rmk}
	The method in Lemma~\ref{lem_lem_qbinomial} can be used to easily get the following generalization.
	\begin{equation}\label{eq_lem_qbinomial_gener}
	\left[\substack{N_1 + N_2 + \cdots + N_k \\ \\  t } \right] = \sum_{i_1 +i_2 \cdots + i_k = t} q^{\left(\sum_{j=1}^{k-1} N_j i_{j+1} - \sum_{ \substack{ 1 \leq p < q \leq k} } i_p i_q\right)} 	\left[\substack{N_1 \\ \\  i_1 } \right]  	\left[\substack{N_2 \\ \\  i_2} \right] \cdots \left[\substack{N_3 \\ \\  i_k} \right].
	\end{equation}
\end{rmk}
Using the results proved above by the method of tiling, Jacobi triple product identity is now proved.
 \begin{thm} [Jacobi triple product identity] \label{thm_Jacobi}
	If $ z\neq 0 $ and $ |q| <1 $, then
	\begin{equation}  \label{eq_thm_Jacobi}
	\prod_{n=0}^{\infty} (1 + z q^{n+1}) (1 + z^{-1} q^{n}) (1-q^{n+1} )  = \sum_{n = -\infty}^{\infty} z ^n q^{\frac{n^2 +n}{2}} .
	\end{equation}
\end{thm}
\begin{proof}
	We will first prove the following finite version of the equality, after which the Jacobi triple product identity will follow by letting $ N $ approach to $ \infty $ in the limit.
	\begin{equation}\label{eq_both_terms}
	\prod_{i=1}^{N} (1 + z q^i ) (1 + z^{-1} q^{i-1}) = \sum_{m=-N}^{N} z^{m} q^{\frac{m^2+m}{2}} \left[\substack{2N \\ \\ N + m } \right]. 
	\end{equation}
	This idea was used by Cauchy to prove the Jacobi triple product identity. The result in \eqref{eq_both_terms} follows from q-binomial theorem (Theorem \ref{thm_bm}, which was proved using the tiling method) by replacing $ N  $ and $ z $ by $ 2N $ and $ z q^{-N} $ respectively. 
	We note that $$ \lim\limits_{ N \to \infty} \left[\substack{2N \\ \\ N + m }\right] =  \lim\limits_{ N \to \infty} \, \frac{(q)_{2N}}{(q)_{N+m} (q)_{N-m}} = \frac{1}{(q)_{\infty}}. $$ Therefore, as $ N \to \infty  $ in \eqref{eq_both_terms},  using Tannery’s theorem (Sect.~3.7.1, \cite{loya2017amazing}) to change the order of limit and the summation, we obtain 
	\begin{align}\label{eq_thm_final}
&	\lim\limits_{ N \to \infty} \prod_{i=1}^{N} (1 + z q^i ) (1 + z^{-1} q^{i-1}) = \sum_{m=- \infty}^{\infty}  \frac{z^{m} q^{\frac{m^2+m}{2}}}{(q)_{\infty}} \\ & \implies   (q)_{\infty} \prod_{n=0}^{\infty} (1 + z q^{n+1}) (1 + z^{-1} q^{n})  = \sum_{m=-\infty}^{\infty}  \, z^{m} q^{\frac{m^2+m}{2}}. \nonumber
	\end{align}
	This completes the proof of the theorem. \\
	
	\noindent \textbf{Alternate proof:} We now give another tiling proof of \eqref{eq_both_terms}.
	Consider two $ 1 \times N $ boards $ F $ and $ G $.  Let the weight function of $ F $ be $ F_w(i) = z^{-1} q^{i-1}  $ and the weight function of $ G $ be
	$ G_w(i) = z q^i $. Let $ F $ be placed on the top of $ G $, giving a $ 2 \times N $ board as shown in Fig.\!~\ref{fig_wieght_boards_AB}.
	We consider all possible tilings of both $ F $ and $ G $, and denote the combined board as $ F\star G $. Clearly the associated generating function is
	\begin{equation}\label{eq_gen_fun}
	H(F\star G, z,q) = \prod_{i=1}^{N} (1 + z q^i ) (1 + z^{-1} q^{i-1}) = \sum_{ \substack{\text{tilings} \ T \ \text{of F} \\ \text{tilings}	\ S \ \text{of G} }} W(T)W(S).
	\end{equation}
	If we start with $ n $ black tiles, we can use $ i $ of it to place on board $ F $ and $ j $ of it to place on board $ G $ with $ i +j = n $. Let us consider all the tilings such that $ |i - j| = m  $, i.e., all the tilings such that the absolute difference between the black tile on $ F $ and $ G $ is exactly $ m $ black tiles.  We have,
	\begin{align} \label{eq_gen_fun_two}
	H(F\star G, z,q )& = \sum_{m=0}^{N} \sum_{i=0}^{N-m} F(i+m) G(i) +  \sum_{m=1}^{N}  \sum_{i=0}^{N-m} F(i) G(i+m),
	\end{align}
	where $ F(m) $, $ G(m) $, denote the generating functions for all the tilings of with $ m $ black tiles of $ F $, $ G $ respectively  (c.f., Definition \ref{def_bmk_bm}, \eqref{eq_b_m}).
	Using  \eqref{eq_A_m}, the first term of the sum on the right is computed as follows.
	\begin{align}\label{eq_first_term}
	\sum_{m=0}^{N} \sum_{i=0}^{N-m} F(i+m) G(i) & = \sum_{m=0}^{N} \sum_{i=0}^{N-m}\, z^{-(i+m)} q^{\frac{(i+m)(i+m-1)}{2}} \left[\substack{N \\ \\ i+m } \right]  z^{i} q^{\frac{i^2 +i}{2}}   \left[\substack{N \\ \\  i } \right] \nonumber \\
	& = \sum_{m=0}^{N} z^{-m} q^{\frac{m^2-m}{2}}\sum_{i=0}^{N-m}\,  q^{i(i+m)} \left[\substack{N \\ \\ (N -m) -i } \right] \left[\substack{N \\ \\  i } \right] \nonumber \\
	& = \sum_{m=0}^{N} z^{-m} q^{\frac{m^2-m}{2}} \left[\substack{2N \\ \\ N - m } \nonumber \right] \qquad \text{(using Lemma \ref{lem_lem_qbinomial})}\\
	& = \sum_{m=-N}^{0} z^{m} q^{\frac{m^2+m}{2}} \left[\substack{2N \\ \\ N + m } \right].
	\end{align}
	A similar computation for the second term of the sum yields
	\begin{equation}\label{eq_second_term}
	\sum_{m=1}^{N}  \sum_{i=0}^{N-m} F(i) G(i+m) = \sum_{m=1}^{N} z^{m} q^{\frac{m^2+m}{2}} \left[\substack{2N \\ \\ N + m } \right]. 
	\end{equation}
	From \eqref{eq_gen_fun_two}, \eqref{eq_first_term} and \eqref{eq_second_term} we get \eqref{eq_both_terms}. Then, as before, when $ N$  approaches $\infty  $,  Jacobi triple product identity follows from \eqref{eq_both_terms}.
\end{proof}

\section{Rogers-Ramanujan Identities}\label{Sec_Ramanujan}
In this section we will prove the celebrated Rogers-Ramanujan identities, given below, by the tiling method.  
\begin{align}\label{eq_Roger_Ramanujan}
\sum_{m=0}^{\infty} \, \frac{q^{m^2}}{(q)_m} &= \prod_{m=0}^{\infty} \, \frac{1}{(1-q^{5m+1})(1-q^{5m+4})} \\
\sum_{m=0}^{\infty} \, \frac{q^{m^2+m}}{(q)_m} &= \prod_{m=0}^{\infty} \, \frac{1}{(1-q^{5m+2})(1-q^{5m+3})}. 
\end{align}

The idea of the proof is very similar to the Cauchy's idea used in proving Jacobi triple product identity. We will closely follow the proof given in \cite{bressoud1983easy}, which relies on a key lemma, namely Lemma 1 in \cite{bressoud1983easy}. We will define the concept of \emph{level of a tiling}, which will turn out to be the key step in the proof of Theorem \ref{thm_rogers_ramanujan_main}. Then Lemma 1  in \cite{bressoud1983easy}, will be a simple corollary of Theorem \ref{thm_rogers_ramanujan_main}.
\begin{rmk}
	The concept of level of a tiling is new and it may be useful in other contexts.
\end{rmk}

\begin{figure}[ht]
	\begin{center}
		\begin{tikzpicture} [scale =1]	
			\foreach \Z in {1,2,3,4}
			{
				\draw[thick] (\Z,0) -- (\Z + 1,0) -- (\Z + 1,1) -- (\Z,1) -- cycle;
				\node at (\Z + 0.5,-0.4) {$ \Z$};
			}
			\foreach \Z in {2,3,4}{
				\pgfmathtruncatemacro{\i}{\Z-1};
				\node at (\Z + 0.5,0.5) {$ z_1^{\i}  $};
			}
			
			\foreach \x in {5,6,7}
			\draw[thick] (\x,0) -- (\x + 1,0) -- (\x + 1,1) -- (\x,1) -- cycle;
			
			\foreach \x in {8,9}
			\draw[thick] (\x,0) -- (\x + 1,0) -- (\x + 1,1) -- (\x,1);
			
			\node at (5 + 0.5,0.5) {$ \cdots$};
			\node at (6 + 0.5,0.5) {$ z_1^i $};
			\node at (7 + 0.5,0.5) {$ \cdots$};
			\node at (6 + 0.5,-0.4) {$ i+1$};
			\node at (8 + 0.5,0.5) {$ \cdots $};
			\node at (9 + 0.5,0.5) {$ \cdots $};

			\foreach \Z in {1,2,3,4}
			{   \pgfmathtruncatemacro{\i}{\Z-1};
				\draw[thick] (\Z,1) -- (\Z + 1,1) -- (\Z + 1,2) -- (\Z,2) -- cycle;
			}
			
			\foreach \Z in {2,3,4}{
				\pgfmathtruncatemacro{\i}{\Z-1};
				\pgfmathtruncatemacro{\j}{\i};
				\node at (\Z + 0.5,1.5) {$ z_{2}^{\i}$};
			}
			\foreach \x in {4,5,6,7}
			\draw[thick] (\x,1) -- (\x + 1,1) -- (\x + 1,2) -- (\x,2) -- cycle;
			
			\foreach \x in {8,9}
			\draw[thick] (\x,1) -- (\x + 1,1) -- (\x + 1,2) -- (\x,2);
			
			\node at (5 + 0.5,1.5) {$ \cdots$};
			\node at (6 + 0.5,1.5) {$ z_2^{i} $};
			\node at (7 + 0.5,1.5) {$ \cdots$};
			\node at (8 + 0.5,1.5) {$ {\cdots}$};
			\node at (9 + 0.5,1.5) {$\cdots $};
			\node at (9 + 0.5,-0.4) {$ \infty$};

			\foreach \Z in {1,2,3,4}
			{   \pgfmathtruncatemacro{\i}{\Z-1};
				\draw[thick] (\Z,2) -- (\Z + 1,2) -- (\Z + 1,3) -- (\Z,3) -- cycle;
			}
			
			\foreach \Z in {2,3,4}{
				\pgfmathtruncatemacro{\i}{\Z-1};
				\pgfmathtruncatemacro{\j}{3*\i};
				\node at (\Z + 0.5,2.5) {$ z_{3}^{\i}$};
			}
			\foreach \x in {4,5,6,7}
			\draw[thick] (\x,2) -- (\x + 1,2) -- (\x + 1,3) -- (\x,3) -- cycle;
			
			\foreach \x in {8,9}
			\draw[thick] (\x,2) -- (\x + 1,2) -- (\x + 1,3) -- (\x,3);
			
			\node at (5 + 0.5,2.5) {$ \cdots$};
			\node at (6 + 0.5,2.5) {$ z_3^{i} $};
			\node at (7 + 0.5,2.5) {$ \cdots$};
			\node at (8 + 0.5,2.5) {$ {\cdots}$};
			\node at (9 + 0.5,2.5) {$\cdots $};
			
			\node at (0.4,0.4) {$ 1$};
			\node at (0.4,1.4) {$ 2$};
			\node at (0.4,2.4) {$ 3$};
			\foreach \y in {0,1,2}{
				\node at (1.4,\y+0.5) {$1$};}
			
			
		\end{tikzpicture}
	\end{center}
	\caption{A $ 3 \times \infty $ board with its weight function $ B_w(i+1,j) = z_j^i $ for $ i \in \NN $ and  $ j \in \{1,2,3\}$.}\label{fig_wieght_boards_three_infinity_thm}
\label{fig_wieght_boards_three_infinity}
\end{figure}

\begin{figure}
	\begin{center}
		\includegraphics[scale=1.2]{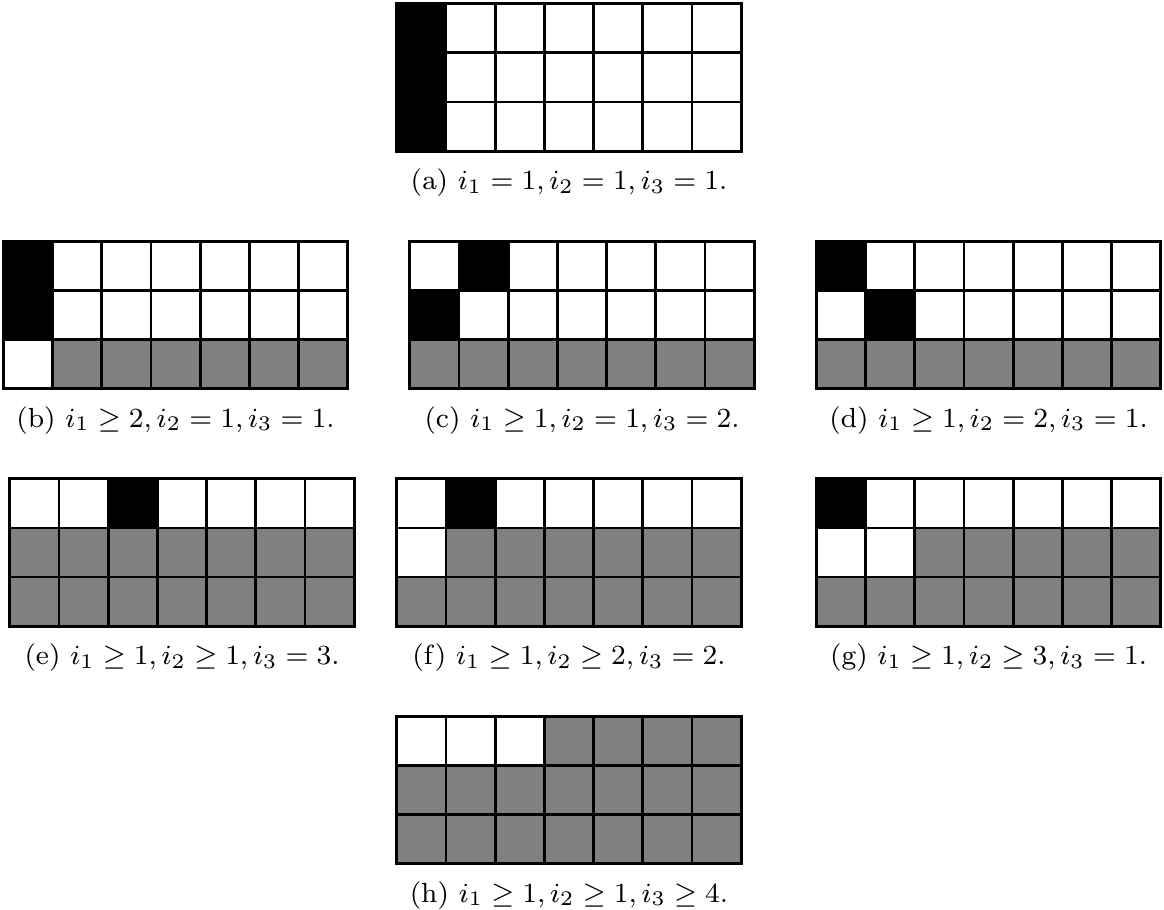}
	\end{center}
	\caption{Levels for a $ 3 \times \infty  $ board is illustrated, with the level of the tiling increasing from $ 0 $ in (a) to $ 3 $ in (h).  The level is $ 1 $ in (b), (c) and (d), and the level is $ 2 $ in (e), (f) and (g). Note that $ [i_1,i_2,i_3] $ denotes the position of the black tiles (i.e., the position of the black tile in each row). The gray tiles show a range of possible positions of a black tile in a row.}    \label{fig_example_level}
\end{figure}

We will consider a board $ B $ of the size $ N \times \infty $. The weight at the $ (i+1,j) $-th position of $ B $ is $z_j^i $. In Fig.\!~\ref{fig_wieght_boards_three_infinity}, for $ N=3 $, such a board and its weight distribution is shown. 

In this section \emph{a tiling of the board $ B $} will mean placing $ N $ black tiles on the board such that \emph{each row contains exactly one black tile}. We also assume that the rows are counted from the bottom.
\begin{defn}
	
	For a tiling $ T $ of $ B $, its local weight $ W_T(i,j) $ at the position $ (i,j) $ is defined as 
	\begin{equation}\label{eq_WT_i_j}
	W_T(i,j) =  \begin{cases} B_w(i,j)  & \text{if the tile at  $ (i,j) $-th poition is black, i.e., if $T(i,j) = black$,} \\
	1 & \text{if the tile at $ (i,j) $-th position is white, i.e., if $T(i,j) = white$.} \end{cases}
	\end{equation}
	The total weight of a tiling $T$ is defined as 
	\begin{equation}\label{eq_weight_i_j}
	W(T) = \prod_{i=1}^{\infty} \prod_{j=1}^{N} W_T(i,j),
	\end{equation}
	the product of its local weights at all the positions on the board. 
\end{defn}
Since, the weight of a tiling $ T $ depends upon the position of the black tiles in each row, we will find it convenient to say that  \emph{$ T $ is a tiling with black tiles at the position $ [i_1,i_2, \ldots, i_N] $} if in the $ k $-th row there is a black tile at the position $ i_k $, i.e., there is a black tile at the  $ (i_k,k) $-th position for each $ 1 \leq k \leq N$.

Now we define the \textit{level of a tiling}, which is a non-negative integer associated to each tiling $ T $ of the board $ B $.
\begin{defn} [Level of a tiling]\label{def_level}
	Let $ T $ be a tiling with black tiles at the position $ [i_1,i_2, \ldots, i_N] $. If $ i_1=i_2= \cdots =i_N =1 $ then the level of $ T $ is defined to be $ 0 $.
	Otherwise, we say that $ T $ is of level $ j $, with $ 1 \leq j \leq N$, if the following conditions are satisfied
	\begin{enumerate}
		\item A positive integer $ s $ exists such that $ 1 \leq s \leq i_j $ and $ s + \sum_{n=j+1}^{N} i_{n} = N+1 $, 
		\item $ 1 \leq i_k  \leq j+1 \quad \text{ for } j <  k \leq N $.
		\end{enumerate}
\end{defn}

For a tiling $ T  $ of level $ j $, we note that there are no restrictions on row $ 1 $ to $ j-1 $, i.e., $ \ds 1 \leq i_k  \leq \infty, \text{ for } 1 \leq k < j$, and in each of these rows one black tile can be placed at any arbitrary position. However, the above definition of level $ j $, imposes a restriction  on the placement of tiles from $ j $-th row to $ N $-th row. For $ N=3 $, this is illustrated in Fig.\!~\ref{fig_example_level}, with the assumption that the board in each sub-figure, and also the gray tiles, extend to infinity on the right. The gray tiles show a range of possible positions of a black tile in a row. Of course, for any particular tiling a black tile can occupy only a unique position in each row. 

It is a consequence of the definition of the level in Definition \ref{def_level}, that the following holds for a tiling $ T $ of level $ j $,
\begin{equation*}
N-j \leq \sum_{k=j+1}^{N} i_{k} \leq N. 
\end{equation*}

It is easy to see that each tiling $ T $ will have a unique level $ j $ with $ 0 \leq j  \leq N $. Suppose $ T $ has two distinct levels $ p $ and $ q $ with $ p < q $.  Let $ t = \sum_{k=p+1}^{q-1} i_k  $, $ t_1 = \sum_{k=q+1}^{N} i_k $. Then there must exist positive integers $ s \geq 1$ and $ s_1 \geq 1 $ such that, $ s + t + i_q + t_1 = N+1 $, and $ s_1 + t_1 = N+1 $, and $ i_q \geq  s_1  $. But it means, $ i_q + t_1 \geq N+1 $, and $ s + t \leq 0 $, a contradiction since $ s \geq 1 $ and $ t \geq 0 $. It proves that any tiling $ T $ can not have two different levels.

Next, we show that each tilling has a level. Let $ T $ be a tiling with black tiles at the position $ [i_1,i_2, \ldots, i_N] $. First, we can assume that there exists a minimum $ k \geq 0  $, with the property that 
\begin{equation}\label{eq_property_min}
i_{n} \leq k +1, \qquad \text{for } k+1 \leq n  \leq N.
\end{equation}
If this is not the case then on taking $ k=N-1 $, it would mean $ i_N \geq N+1 $ and $ T $ will be of level $ N $, and we will be done. Also, it is easy to check that if $ k=0 $, then $ T $ is of level $ 0 $. Therefore, we can assume that a minimum $ k > 0 $ exists, such that it  has the property given in \eqref{eq_property_min}. Consider the equation $ s_k + \sum_{n=k+1}^N \, i_n = N+1 $. If $ \sum_{n=k+1}^N \, i_n \geq  N+1  $, then there exists $ r $ such that $ k + 1 \leq r \leq N$ and $ s_{r} + \sum_{n=r+1}^N \, i_n =N+1 $, with $ 1 \leq s_{r} \leq i_{r}  $, and we are done as then $ T $ is of level $ r $. So, we can assume $ N-k \leq  \sum_{n=k+1}^N \, i_n \leq N $. If $ \sum_{n=k+1}^N \, i_n = N $, then  $ s_k + \sum_{n=k+1}^N \, i_n = N+1   $ has a solution $ s_k $, such that $ 1 = s_k \leq i_k $, and $ T $ is of level $ k $. Therefore,  we can assume $  N-k \leq  \sum_{n=k+1}^N \, i_n \leq N -1 $.   This means,
\begin{equation}\label{eq_limit_range_i_k}
1 \leq i_n \leq k, \qquad \text{for } k+1 \leq n  \leq N.
\end{equation}
Since, $  N-k \leq  \sum_{n=k+1}^N \, i_n \leq N -1 $, we note that $ s_k + \sum_{n=k+1}^N \, i_n = N+1   $ has a solution $ s_k $ such that  $ 2 \leq s_k \leq  k+1$. If $ i_k \geq s_k  $, then $ T $ is of level $ k $ and we are again done. Therefore, assume $ i_k < s_k \leq k+1 $, i.e., $ i_k \leq k $. However, if this is the case then, in view of \eqref{eq_limit_range_i_k}, we get a contradiction to the minimality of $ k $, as on replacing $ k $ by $ k-1 $, the property given in \eqref{eq_property_min} is satisfied. Hence, we are done again. We have shown that every tiling $ T $ must have a level. 

Next, we prove the following result using the concept of the level of a tiling.
\begin{thm}	\label{thm_rogers_ramanujan_main}
	Let $ z_j \in \CC $ with  $ |z_j| < 1  $ for $ 1 \leq j \leq N $. Then,
	\begin{equation}\label{eq_rogers_ramanujan_main_thm}
	\prod_{j=1}^{N} \frac{1}{1 - z_j} = 1 + \sum_{j=1}^{N} \, \,  \frac{h_j(z_j, z_{j+1}, \cdots, z_N) }{(1-z_1) (1-z_2) \ldots (1-z_j)},
\end{equation}
where, 
\begin{equation}\label{eq_homogeneous_poly}
h_j(z_j, z_{j+1}, \cdots, z_N)	:= \sum_{ \substack{
			0 \leq \,  \beta_j, \beta_{j+1}, \cdots, \beta_N \leq \,  j \\  \sum_{k=j}^{N} \beta_k = j
	}}   \quad  \,  {\prod_{k=j}^{N} z_k^{\beta_k} },
\end{equation}
is the complete homogeneous symmetric polynomial of degree $ j $ in  variables $ z_j,z_{j+1},\cdots, z_N $.
\end{thm}

\begin{proof}
		Let $ B $ be an $ N \times \infty $ board and the weight at its $ (i+1,j) $-th position is $ z_j^{i} $ (see Fig.~\ref{fig_wieght_boards_three_infinity_thm} for an illustration for $ N =3$). 
	On expanding each term in the product on the left side, as in \eqref{eq_prod_series}, it becomes immediately clear that the generating function of all the tilings of the board $ B $ equals the left hand side of \eqref{eq_rogers_ramanujan_main_thm},
	
	\begin{equation}\label{eq_prod_series}
		\prod_{j=1}^{N} \frac{1}{(1 - z_j)}  = 	\prod_{j=1}^{N} \sum_{n=0}^{\infty} \, z_j^n .
	\end{equation}

	In order to find the generating function of all the tiling of the board $ B $, we compute the generating functions of all the tilings of level $ j $, from $ j =0$ to $ j=N $.  By adding these generating functions, we will find the generating function of all the tiling of board $ B $, which will prove the Theorem. 
	For level $ j=0 $ the associated generating function is $ 1 $, and for level $ j=N $ the generating function is $$ \frac{z^N }{\prod_{k=1}^{N} \, (1 -z_k)} .$$
	Therefore, we assume $ 1 \leq j <N $.
	Rows $ 1 $ to $ j-1 $ contribute the following to the generating function
	\begin{equation}\label{eq_contrib_one}
		\prod_{k=1}^{j-1} \sum_{ r =0 }^{\infty} z_k^r  = \prod_{k=1}^{j-1} \, \frac{1}{1 - z_k}.
	\end{equation}
	Let \[
	D_j: = \left\{  i_j, i_{j+1}, \cdots, i_N: \substack{ N-j \, \leq   \sum_{k=j+1}^{N} i_k \, \leq N\\ 
		1 \leq \,  s,i_{j+1}, \cdots, i_N \leq \,  j+1 \\ s + \sum_{k=j+1}^{N} i_k = N+1
	} \right\},
	\]
	denote the set corresponding to all possible values of the positions of black tiles in the rows $  j, {j+1}, \cdots, N $, in all the tilings of level $ j $.
	
	The contributions of rows $ j $ to $ N $ is computed as follow.
	\begin{align}\label{eq_contrib_j_n}
		 \sum_{D_j}  \sum_{ i_{j} = s}^{\infty} (z_j)^{i_j-1}  \quad \prod_{k=j+1}^{N} (z_k)^{i_k-1} 
		&= \sum_{D_j}   \frac{(z_j)^{s-1} }{1-z_j}  \quad \prod_{k=j+1}^{N} (z_k)^{i_k-1} \nonumber \\
		&=  \frac{1 }{1-z_j} \, \cdot \sum_{ \substack{
				0 \leq \,  \beta_j, \beta_{j+1}, \cdots, \beta_N \leq \,  j \\  \sum_{k=j}^{N} \beta_k = j
		}}   \quad  \,  \prod_{k=j}^{N} (z_k)^{\beta_k}.
	\end{align}
	Here in the last step, we have carried out change of variables by replacing $ s$, $i_{j+1}$, $i_{j+2}, \cdots, i_N $ by $ \beta_j +1$, $ \beta_{j+1}+1 $, $ \beta_{j+2}+1 $, $\cdots, \beta_N+1  $, respectively, in the summation. Therefore the contribution to the generating function by all the tilings of level $ j $ is given by 
	\[ \sum_{ \substack{
			0 \leq \,  \beta_j, \beta_{j+1}, \cdots, \beta_N \leq \,  j \\  \sum_{k=j}^{N} \beta_k = j
	}}   \quad  \,  \frac{\prod_{k=j}^{N} z_k^{\beta_k} }{(1-z_1) (1-z_2) \ldots (1-z_j)} =  \frac{h_j(z_j, z_{j+1}, \cdots, z_N) }{(1-z_1) (1-z_2) \ldots (1-z_j)}.  \]
	This completes the proof of the theorem.
	
\end{proof}

\begin{rmk}
	For $ N=2 $ and $ 3 $, Theorem  \ref{thm_rogers_ramanujan_main} gives the following identities
	\begin{align*}
		\frac{1}{(1-z_1)(1-z_2)} &= 1 + \frac{z_1+z_2}{1-z_1} + \frac{z_2^2}{(1-z_1)(1-z_2)},\\
		\frac{1}{(1-z_1)(1-z_2)(1-z_3)} &= 1 + \frac{z_1+z_2+z_3}{1-z_1} + \frac{ z_2^2 + z_2 z_3 + z_3^2}{(1-z_1)(1-z_2)} + \frac{z_3^3}{(1-z_1)(1-z_2)(1-z_3)}.
	\end{align*}
\end{rmk}

As an immediate corollary of Theorem \ref{thm_rogers_ramanujan_main}, we obtain the following result (which is Lemma 1 in \cite{bressoud1983easy}).

\begin{cor}\label{cor_rogers_ramanujan_main}
	\begin{equation}\label{eq_rogers_ramanujan_main}
		\prod_{j=1}^{N} \frac{1}{(1 - zq^j)} = 1 + \sum_{j=1}^{N} \frac{z^j q^{j^2}}{(1-zq)\cdots (1-zq^{j})} \, \left[\substack{N \\ \\  j} \right].
	\end{equation}
\end{cor}
\begin{proof}
On setting $ z_j = z q^j $ in \eqref{eq_rogers_ramanujan_main_thm}, for $ j=1 $ to $ N $, we obtain
\begin{align*}
	\prod_{j=1}^{N} \frac{1}{(1 - zq^j)} &=	 1 + \sum_{j=1}^{N} \, \, \,  \sum_{ \substack{
			0 \leq \,  \beta_j, \beta_{j+1}, \cdots, \beta_N \leq \,  j \\  \sum_{k=j}^{N} \beta_k = j
	}}   \quad  \,  \frac{\prod_{k=j}^{N} (z q^k)^{\beta_k} }{(1-z q) (1-z q^2) \ldots (1-z q^j)} \\
&=	 1 + \sum_{j=1}^{N} \, \, \,  \sum_{ \substack{
		0 \leq \,  \beta_j, \beta_{j+1}, \cdots, \beta_N \leq \,  j \\  \sum_{k=j}^{N} \beta_k = j
}}   \quad  \,   \frac{ z^{ \sum_{k=j}^{N} \beta_k} (q^j)^{ \sum_{k=j}^{N} \beta_k} q^{{\sum_{k=1}^{N-j}k \beta_{j+k}}}}{(1-z q) (1-z q^2) \ldots (1-z q^j)} \\
 &=	 1 + \sum_{j=1}^{N} \, \, \, \frac{z^j q^{j^2}}{(1-z q) (1-z q^2) \ldots (1-z q^j)} \, \,  \sum_{ \substack{
		0 \leq \,  \beta_j, \beta_{j+1}, \cdots, \beta_N \leq \,  j \\  \sum_{k=j}^{N} \beta_k = j
}}   \quad  \, q^{ \,  {\sum_{k=1}^{N-j}k \beta_{j+k}} }\\
&=	 1 + \sum_{j=1}^{N} \, \, \, \frac{z^j q^{j^2}}{(1-z q) (1-z q^2) \ldots (1-z q^j)} \, \,    \left[\substack{N \\ \\  j} \right]. 
\end{align*}
Here the last summation equals $ \left[\substack{N \\ \\  j} \right] $ by Corollary \ref{Cor_partition_interpretation}.
\end{proof}
The next result was known to Cauchy.  
\begin{cor}
	\begin{equation}\label{cor_Cauchy}
		\frac{1}{(zq;q)_{\infty}}  =  \sum_{j=0}^{\infty} \, \frac{z^j q^{j^2}}{ (zq;q)_j (q)_j }.
	\end{equation}
\begin{proof}
Take $ N \to \infty $ in Corollary  \ref{cor_rogers_ramanujan_main} and	use Tannery's theorem to change the order of limits and the summations.
\end{proof}
\end{cor}
The following lemma is Lemma 2 in \cite{bressoud1983easy}. 
\begin{lem}\label{lem_eq_lemma_2}
	For a positive integer $ n $ and any real number $ a $, we have
	the following identity.
	\begin{equation}\label{eq_lemma_2}
	\sum_{m=-n}^{n} \frac{z^m q^{a m^2} }{(q)_{n-m} (q)_{n+m}} =\sum_{s=0}^n \frac{q^s}{(q)_{n-s}} 	\sum_{m=-s}^{s} \frac{z^m q^{(a-1)m^2} }{(q)_{s-m} (q)_{s+m}}.
	\end{equation}
\end{lem}

\begin{proof}
	The proof follows by expanding $ \frac{1}{(q)_{n+m}} $ using Corollary~\ref{cor_rogers_ramanujan_main} and substituting the resulting expression in left side of \eqref{eq_lemma_2}, and carrying out certain algebraic manipulations. See the proof of Lemma 2, \cite{bressoud1983easy}, for complete details.
\end{proof}

The Rogers-Ramanujan identity can now be obtained by using Lemma \ref{lem_eq_lemma_2}, exactly following \cite{bressoud1983easy}.
\begin{thm} [Rogers-Ramanujan Identities]
	\begin{align}
	\sum_{m=0}^{\infty} \, \frac{q^{m^2}}{(q)_m} &= \prod_{m=0}^{\infty} \, \frac{1}{(1-q^{5m+1})(1-q^{5m+4})}  \label{eq_Roger_Ramanujan_one}\\
	\sum_{m=0}^{\infty} \, \frac{q^{m^2+m}}{(q)_m} &= \prod_{m=0}^{\infty} \, \frac{1}{(1-q^{5m+2})(1-q^{5m+3})} \label{eq_Roger_Ramanujan_two}.
	\end{align}
\end{thm}

\begin{proof}
	We simplify the right side of \eqref{eq_Roger_Ramanujan_one} using Jacobi triple product identity as follows,
\begin{align}
\prod_{m=0}^{\infty} \, \frac{1}{(1-q^{5m+1})(1-q^{5m+4})} & = \frac{1}{(q)_{\infty}}\prod_{m=0}^{\infty} \, \frac{1- q^{m+1}}{(1-q^{5m+1})(1-q^{5m+4})} \nonumber \\
&=  \frac{1}{(q)_{\infty}}\prod_{m=1}^{\infty} \, (1-q^{5m}) (1-q^{5m-2})(1-q^{5m-3}) \nonumber \\
&=   \frac{1}{(q)_{\infty}}\sum_{m=-\infty}^{\infty} \, (-1)^m q^{\frac{5m^2 +m}{2}} \,\,\, \text{(from triple product identity)} \nonumber \\
&= \lim\limits_{N \to \infty} \sum_{m=-N}^{N} \, (-1)^m q^{\frac{5m^2+m}{2}}   \left[\substack{2N \\ \\  N-m} \right], \label{eq_proof_Roger_one}
\end{align}
where at the last step we have employed Tannery’s theorem, as we did in our proof of Jacobi triple product identity.
Next, at the first step in the following, we apply Lemma \ref{lem_eq_lemma_2} twice to reduce the coefficient of $ m^2 $ to $ \frac{1}{2} $, and then in the next step we use \eqref{eq_both_terms} with $ z=-1 $ to reduce the inner sum to a product.
\begin{align*}
\sum_{m=-N}^{N} \,\frac{(-1)^m q^{\frac{5m^2+m}{2}}}{(q)_{N-m} (q)_{N+m}}   &=  \sum_{s=0}^N \frac{q^{s^2}}{(q)_{N-s}} \sum_{t=0}^{s} \frac{q^{t^2}}{(q)_{s-t}} \sum_{m=-t}^{t} \, (-1)^m q^{\frac{m^2+m}{2}}     \left[\substack{2t \\ \\  t-m} \right] \, \frac{1}{(q)_{2t}} \\
 &=  \sum_{s=0}^N \frac{q^{s^2}}{(q)_{N-s}} \sum_{t=0}^{s} \frac{q^{t^2} }{(q)_{s-t}} \prod_{m=1}^{t} \, (1 -q^m) (1- q^{m-1}) \,  \frac{1}{(q)_{2t}} \\
&=   \sum_{s=0}^N \frac{q^{s^2}  }{(q)_{s} (q)_{N-s}}.
\end{align*}
Here the last step follows from the fact that the product on the right, in the previous step, survives only for $ t=0 $. Therefore,
\begin{align*}
\lim\limits_{N \to \infty} \sum_{m=0}^{N} \, (-1)^m q^{\frac{5m^2+m}{2}}   \left[\substack{2N \\ \\  N-m} \right] = \lim\limits_{N \to \infty}   \sum_{s=0}^N \frac{q^{s^2} \, (q)_{2N} }{(q)_{s} (q)_{N-s}}  = 	\sum_{s=0}^{\infty} \, \frac{q^{s^2}}{(q)_s}.
 \end{align*}
This concludes the proof of first identity \eqref{eq_Roger_Ramanujan_one}. Similarly, one can obtain the following using Jacobi triple product identity
\begin{equation*}
\prod_{m=0}^{\infty} \, \frac{1}{(1-q^{5m+2})(1-q^{5m+3})}  =  \frac{1}{(q)_{\infty}}\sum_{m=0}^{\infty} \, (-1)^m q^{\frac{5m^2 + 3m}{2}},
\end{equation*}
and then proceeding just like before the second identity \eqref{eq_Roger_Ramanujan_two} can be proved.

\end{proof}

\begin{rmk}
	Since, the following more general Theorem from \cite{bressoud1983easy} is proved using Lemma 1 and Lemma 2 in \cite{bressoud1983easy}, which are our Lemma \ref{cor_rogers_ramanujan_main} and \ref{lem_eq_lemma_2} respectively, we have essentially provided a tiling proof of this result.
	\begin{align*}\label{eq_thm_bresoud}
	\sum_{s_1,\cdots , s_k }  & \frac{q^{s_1^2 + \cdots s_k^2}}{(q)_{N- s_1} (q)_{s_1 -s_2} (q)_{s_{k-1} - s_{k-2}} (q)_{2s_k}} \prod_{1}^{s_k} (1 + z q^m ) (1 + z^{-1}q^{m-1}) \\
& 	= (q)_{2N}^{-1} \sum_{m} (-1)^m z^m q^{((2k+1)m^2 + m)/2}  \left[\substack{2N \\ \\  N-m} \right].
 	\end{align*}
\end{rmk}

\section{Proving some identities of Rogers using dominoes and squares} \label{Sec_Rogers}

Consider tilings of a $ 1 \times \infty $ board $ F $ using dominoes and white squares. Such tilings were considered in \cite{MR2587022}. 
 Although our approach is very similar to \cite{MR2587022}, our method differs from  \cite{MR2587022} in that we do not use projections of dominoes, 
 instead we obtain recurrence relations just as we did in the previous sections. 
In this section, any tiling of the board will consist of a finite number of dominoes and an infinite number of white squares. A domino for us will be a pair of tiles `glued' together such that the left tile in the pair is a black tile and the right tile in the pair is a white tile. 
 We say that a domino is placed at the $ k $-th position on the board, if the domino is placed on the board such that the `black tile' of the domino
is at the $ k $-th position and the `white tile' of the domino is at the $ (k+1) $-th position. We can use similar definitions of local weight, total weight, and generating functions of tilings as given in Section \ref{sec_basic_setup} (see Def.~\ref{def_local_weight} through Def.~\ref{def_generating_fun}.) For example, let the associated weight function for the board $ F $  be given by $ F_w(i) = z q^i $. Then for a tiling $ T $ of $ F $, its local weight $ W_T(i) $ at the position $ i $ is defined as 
\begin{equation}\label{eq_local_weight_dominoes}
	W_T(i) :=  \begin{cases} F_w(i)=zq^i,  & \text{if a domino (or equivalently its `black tile') is at the $ i $-th position, } \\
		1, & \text{if the tile at the $ i $-th position is white.} \end{cases}
\end{equation}

Next, we prove the following theorem using the tiling of the board $ F $ by dominoes and white squares.

\begin{thm}\label{thm_theorem_rogers}
	\begin{equation}\label{theorem_rogers}
		\sum_{m=0}^{\infty}	\frac{z^m q^{m^2}}{ (q)_m} =  (-zq^2;q^2)_{\infty} 	\sum_{m=0}^{\infty} \frac{z^m q^{m^2}}{(q^2;q^2)_{m} (-zq^2;q^2)_{m} }
	\end{equation}
\begin{equation}\label{theorem_rogers_two}
	\sum_{m=0}^{\infty}	\frac{z^m q^{m^2+m}}{ (q)_m} =  (-zq^2;q^2)_{\infty} 	\sum_{m=0}^{\infty} \frac{z^m q^{m^2+2m}}{(q^2;q^2)_{m} (-zq^2;q^2)_{m} }.
\end{equation}
\end{thm}
\begin{proof}
	Let $ F $ be the $ 1 \times \infty $ board with its associated weight function $ F_w(i) = z q^i $, as described earlier in the beginning of this section. 
	Let $ F(k,m) $ be the generating function associated to all possible tilings of the board $ F $ with exactly $ m $ dominoes, such that the tilings begin at the $ k $-th  position on the board (i.e., all the positions $ 1  $ through $ k-1 $ have white tiles.)

	On using an argument, similar to the one used to obtain \eqref{eq_euler_recur_two} and \eqref{eq_euler_recur_general}, we get
	\begin{equation}\label{eq_recur_roger_one_two}
		F(2,m) = q^{m} F(1,m), \qquad F(k+1,m) = q^{km} F(1,m).
	\end{equation}
	Since the first tile in any tiling could be either a domino or a white square, we have
	\begin{equation}\label{eq_recur_dominoes}
		F(1,m) = F(2,m) + zq F(3,m-1) = q^m F(1,m) + zq\cdot q^{2(m-1)} F(1,m-1). 
	\end{equation}
	It follows that
	\begin{equation}\label{eq_recur_dominoes_final}
		F(1,m) = \frac{z q^{2m-1}}{1-q^m} F(1,m-1).
	\end{equation}
Clearly $ F(1,0) =1 $. Therefore, we get
	\begin{equation}\label{eq_fonem}
		F(1,m) = \frac{z^m q^{m^2}}{ (q)_m}.
	\end{equation}
	
	 Let us call a domino  an `odd-domino', if it is placed at an odd position on the board. 
	Let $ G(k,m) $ denote the generating function associated to all possible tilings containing exactly $ m $ odd-dominoes with the condition that the tilings  begin at the $ k $-th  position on the board (i.e., the positions $ 1  $ through $ k-1 $ have white tiles.)
	
	To prove \eqref{theorem_rogers} we will use the following equality
	\begin{align}\label{eq_rogers_gen_one}
		\sum_{m=0}^{\infty} \, F(1,m) = \sum_{m=0}^{\infty} \, G(1,m), 
	\end{align}
which is clearly true because both the left and the right sides represent the generating function of all the possible tilings of the board $ F $. 

The expression for $F(1,m)$ was already obtained  in \eqref{eq_fonem}. Now $ G(1,m) $ needs to be computed. First, we note that
	\begin{equation}\label{eq_roger_dominoe_recur_zero}
		G(1,0) = \prod_{m=1}^{\infty} (1 + z q^{2m}) = (-zq^2;q^2)_{\infty},
	\end{equation}
because, this is the case with no odd-dominoes, and at each even position $ 2m $ one can place either a square or a domino and the factor $ (1 + zq^{2m}) $ represents this choice.
	
	In any tiling, the first tile can either be a domino or a white square. Therefore, we also get the following recurrence relations
	\begin{equation}\label{eq_roger_dominoe_recur_three}
		G(1,m) = G(2,m) +  zq  G(3,m-1), \qquad 	G(2,m) = G(3,m) +  zq^2  G(4,m).
	\end{equation}
	Consider any tiling $ T $ in $G(1,m) $ (i.e., $T$  is a tiling with $m$ odd-dominoes starting at the first position on the board). We describe how to shift tiles in $ T $ to get a tiling in $ G(2,m) $, going from the right to the left, starting from the right-most odd-domino. If any odd-domino in $ T $ is followed by two white tiles, then that odd-domino is shifted to the right by two units. If the odd-domino is followed by a white square, thereafter $ r $ even-dominoes, and after that by a white square, then the first odd-domino and the last $ r$-th even-domino both are shifted to the right by one unit. This shifting procedure keeps the number of odd-dominoes invariant, before and after the shifting. The weight of the tiling is increased by a factor of $ q^2$ for each odd-domino. Since there are $ m $ odd-dominoes, it follows that
	\begin{equation}\label{eq_roger_dominoe_recur_one}
		G(2,m) = q^{2m} G(1,m), \qquad 	G(4,m) = q^{2m} G(3,m).
	\end{equation}
	From equations \eqref{eq_roger_dominoe_recur_three} and \eqref{eq_roger_dominoe_recur_one} it follows that
	\begin{equation}\label{eq_roger_dominoe_recur_five}
		G(1,m) = \frac{zq^{2m-1}}{(1 - q^{2m}) (1 + z q^{2m}) } G(1,m-1).
	\end{equation}
From  \eqref{eq_roger_dominoe_recur_zero} and \eqref{eq_roger_dominoe_recur_five}	we obtain,
	\begin{equation}\label{eq_roger_dominoes_recur_six}
		G(1,m) =  (-zq^2;q^2)_{\infty} \frac{z^m q^{m^2}}{(q^2;q^2)_{m} (-zq^2;q^2)_{m} }.
	\end{equation}
Now \eqref{theorem_rogers} follows from \eqref{eq_fonem}, \eqref{eq_rogers_gen_one} and \eqref{eq_roger_dominoes_recur_six}.
Similarly, considering the equality of generating functions
	\begin{align}\label{eq_rogers_gen_two}
	\sum_{m=0}^{\infty} \, F(2,m) = \sum_{m=0}^{\infty} \, G(2,m), 
\end{align}
and noting that $ F(2,m) = q^m F(1,m) $ and $ G(2,m) = q^{2m} G(1,m) $, one can easily obtain \eqref{theorem_rogers_two}.
This completes the proof of the theorem.
\end{proof}
\begin{rmk}
Our proof of Theorem \ref{thm_theorem_rogers} unifies the proofs of Theorem 2 and Theorem 4 in \cite{MR2587022}.
\end{rmk}

We get the following first two identities (\eqref{eq_cor_one} and \eqref{eq_cor_two})  by setting $ z=1 $ and $ z =q $ in \eqref{theorem_rogers}. Similarly, on setting  $ z= \frac{1}{q} $ and $ z = 1 $ in \eqref{theorem_rogers_two}, we obtain the next two identities (\eqref{eq_cor_three} and \eqref{eq_cor_four}). 
\begin{cor}
	\begin{align}
			\sum_{m=0}^{\infty}	\frac{q^{m^2}}{ (q)_m} &=  (-q^2;q^2)_{\infty} 	\sum_{m=0}^{\infty} \frac{ q^{m^2}}{(q^4;q^4)_{m} }.  \label{eq_cor_one}\\
			\sum_{m=0}^{\infty}	\frac{ q^{m^2+m}}{ (q)_m} &=  (-q;q^2)_{\infty} 	\sum_{m=0}^{\infty} \frac{q^{m^2+m}}{(q^2;q^2)_{m} (-q;q^2)_{m+1} }.\label{eq_cor_two}  \\
			\sum_{m=0}^{\infty}	\frac{ q^{m^2}}{ (q)_m} &=  (-q;q^2)_{\infty} 	\sum_{m=0}^{\infty} \frac{ q^{m^2+m}}{(q^2;q^2)_{m} (-q;q^2)_{m} }. \label{eq_cor_three}\\
			\sum_{m=0}^{\infty}	\frac{ q^{m^2+m}}{ (q)_m} &=  (-q^2;q^2)_{\infty} 	\sum_{m=0}^{\infty} \frac{ q^{m^2+2m}}{(q^4;q^4)_{m}}.\label{eq_cor_four}
	\end{align}
\end{cor}
We note that the above identities were proved by Rogers. Corollary 3 and Corollary 5 in  \cite{MR2587022} give tiling proofs of these identities. 

\section{Some new recursive identities using tilings of a $ 1 \times \infty $ board}
Consider a $ 1 \times \infty $ board $ F $ with the following weight function
\begin{equation}\label{eq_recur_weight}
F_w(3n + j ) = z_j q^{n+1}, \qquad  \text{for }  n = 0,1,2,3, \cdots,  \text{ and $ j = 1,2,3 $.}
\end{equation}
So, the weights on $ F $ beginning at the first few positions are  $ z_1 q$, $z_2 q$, $z_3 q$, $z_1 q^2$, $z_2 q^2$, $z_3 q^2$, $z_1 q^3$, $z_2 q^3$, $z_3 q^3  \cdots. $

The generating function of all possible tilings of the board $ F $ is 
\begin{equation}\label{eq_gen_fun_F}
H(F,z_1,z_2,z_3,q) = \prod_{m=1}^{\infty} (1 + z_1 q^m) (1 + z_2 q^m) (1 + z_3 q^m).
\end{equation}
Let $ F(k,m) $ be as in the Definition \ref{def_bmk_bm}, i.e., it is the generating function for all the possible tilings of the board $ F $ using $ m $ black tiles starting at the position $ k $.
It is clear that,
\begin{equation*}
F(1,1) = \sum_{i=1}^{\infty} (z_1 + z_2 + z_3)q^i = \frac{q(z_1 + z_2 + z_3)}{1-q}.
\end{equation*}
We will find a recursion to compute $ F(m) := F(1,m) $ for all $ m $. Clearly, we also have
\begin{equation}\label{eq_gen_fun_F_two}
H(F,z_1,z_2,z_3,q) = \sum_{m=0}^{\infty} F(m).
\end{equation}
Therefore, by using \eqref{eq_gen_fun_F} and \eqref{eq_gen_fun_F_two} we will obtain an identity, which will turn out to be our next theorem.
\subsection{A triple product identity}
If a black tile is shifted  $ 3l $ positions to the right from its current position, then its weight will change by a factor of  $q^l $. Suppose a tiling $ T $ contains $ m $ black tiles and each of these $ m $ tiles is shifted to $ 3 l $ positions to the right, then the weight of $ T $ will change by a factor of $ q^{lm} $. We note that during this shifting the relative positions of all the black tiles remain the same, before and after the shift. Therefore, the effect of multiplication by $ q^{lm} $ on $ F(1,m) $ is to give the generating function $ F(3l+1,m) $. Similar arguments apply for $ q^{lm} F(2,m) $ and $ q^{lm}F(3,m) $. Therefore, we obtain the following recurrence relation for $ l=1 $.
\begin{align} 
F(4,m) &= q^{m} F(1,m). \label{eq_recur_one}
\end{align}
Next, we fix the first tile to be a black tile with contribution $ z_1 q $, then the remaining $ m-1 $ black tiles can be placed in tilings whose generating function would be $ F(2,m-1) $. The total contribution in this case is $ z_1q F(2,m-1)  $.  Then we consider the case when the first tile is white. In this case, the generating function of all such tilings is $ F(2,m) $. We can apply the same argument to  $ F(2,m) $ and $ F(3,m) $. Thereby, we get the following additional recurrence relations.
\begin{align}
F(1,m) &= z_1 q F(2,m-1) + F(2,m), \label{eq_recur_four} \\
F(2,m) &= z_2 q F(3,m-1) + F(3,m), \label{eq_recur_five} \\
F(3,m) &= z_3 q F(4,m-1) + F(4,m). \label{eq_recur_six}
\end{align}
We use \eqref{eq_recur_one}, and \eqref{eq_recur_six} to obtain the following.
\begin{equation}\label{eq_recur_seven}
F(3,m) = z_3 q F(4,m-1) + F(4,m) = z_3 q^{m} F(1,m-1) + q^m F(1,m).
\end{equation}
By using  \eqref{eq_recur_one} to \eqref{eq_recur_seven}, we get the following,
\begin{align*}
F(1,m) &= z_1 q F(2,m-1) + F(2,m) \\
& = z_1q (z_2 q F(3,m-2) + F(3,m-1 ) ) + z_2 q F(3,m-1) + F(3,m) \\
& =  z_1 z_2 q^2 F(3,m-2) + (z_1 + z_2)q F(3,m-1) + F(3,m) \\
& =  z_1 z_2 q^2 (z_3 q^{m-2} F(1,m-3) + q^{m-2} F(1,m-2 )) +  \\
& \quad  (z_1 + z_2)q ( z_3 q^{m-1} F(1,m-2) + q^{m-1} F(1,m-1)) + z_3 q^{m} F(1,m-1) + q^{m} F(1,m).
\end{align*}
Therefore,
\begin{align} \label{eq_recur_main_one}
F(1,m) & = \frac{q^{m}}{1-q^{m}} \left( (z_1 + z_2+ z_3) F(1,m-1) + (z_1 z_2 + z_2 z_3 + z_3 z_1) F(1,m-2)\right. \nonumber\\ 
&\left. \qquad  \qquad \qquad + z_1 z_2 z_3 F(1,m-3) \right).
\end{align}
We also have the following
\begin{align*}
F(1,0) &=1, F(1,1) = \frac{q(z_1 + z_2 + z_3)}{1-q}, \, \text{ and }  \\ 
F(1,2) &= \frac{q^2}{(1-q)(1-q^2)} \left(  q(z_1 + z_2 + z_3)^2 + (1-q) (z_1 z_2 + z_2 z_3 + z_3 z_1) \right).  \\ 
\end{align*}

\begin{rmk}\label{rmk_operator}
	The calculation to obtain \eqref{eq_recur_main_one} can be simplified by treating recursive relations in \eqref{eq_recur_one} to \eqref{eq_recur_seven} as operators as follows. Let $ D $ be an operator such that $ \alpha D^n\,  f(m) = \alpha f(m-n) $ for any  function $ f $, and constant $ \alpha \in \CC$, and $ n \in \NN $.  Using $ D $, \eqref{eq_recur_one} can be written as $ (z_1qD + 1) F(2,m) = F(1,m).$  Similarly, we can write,
\begin{align*}
F(1,m)	&= (z_1 q D + 1 ) (z_2 q D + 1 ) (z_3 q D + 1) F(4,m) \\
&= \left( z_1 z_2 z_3 q^3 D^3 + (z_1z_2 + z_2 z_3 + z_3 z_1) q^2 D^2 + (z_1 +z_2 +z_3) qD + 1 \right) q^m F(1,m) \\
& =  z_1 z_2 z_3 q^m F(1,m-3) + (z_1z_2 + z_2 z_3 + z_3 z_1) q^m F(1,m-2) \\ & \quad + (z_1 +z_2 +z_3) q^m F(1,m-1) + q^m F(1,m), 
\end{align*}
from which  \eqref{eq_recur_main_one} follows.
\end{rmk}

For $ m \geq 1 $, let
\begin{equation}
F(1,m)  =  \frac{q^{m} P(m) }{(1-q)(1-q^2)(1-q^3) \cdots  (1-q^{m})},          
\end{equation}
where $ P(m) $ a polynomial in $ q,z_1,z_2 $ and $ z_3 $ to be determined. 
By using \eqref{eq_recur_main_one} one can compute a recurrence for $ P(m)  $. Also, $ P(1), P(2) $ and $ P(3) $ can be explicitly computed.
From the above discussion we have the following result.
\begin{thm}\label{thm_main_new_recur}
	Let $ z_1, z_2, z_3 ,q $ be complex numbers with  $ |q| < 1 $. Then,
	\begin{align*}
\prod_{m=1}^{\infty} (1 + z_1 q^m) (1 + z_2 q^m) (1 + z_3 q^m) &= \sum_{m=0}^{\infty} 	\frac{q^{m} P(m) }{(1-q)(1-q^2)(1-q^3) \cdots  (1-q^{m})}, 
	\end{align*}
	where for $  m > 3$, $ P(m) $ is a polynomial recursively defined as
	\begin{align*} 
	P(m) &= q^{m-3} \left((z_1 + z_2 + z_3) \, q^2 P(m-1) + (z_1 z_2 + z_2 z_3 + z_3 z_1) \, q (1-q^{m-1})  P(m-2)\right. \nonumber \\
	& \qquad\left. + z_1 z_2 z_3 \, (1-q^{m-2})(1-q^{m-1}) P(m-3) \right), \\
	& \text{with } P(0) =1, \, \, P(1) = z_1 + z_2 + z_3, \, \, P(2) =  q(z_1 + z_2 + z_3)^2 + (1-q) (z_1 z_2 + z_2 z_3 + z_3 z_1).
	\end{align*}
\end{thm}
\begin{rmk}
	On setting $ z_1 = z$ and $ z_2=z_3=0 $ in Theorem \ref{thm_main_new_recur}, one can recover Euler's identity given in 
	\eqref{eq_Euler}. One can obtain explicit non-recursive formulas in many special cases, for example when $ z_1 + z_2 =0 $ and $ z_3 =0 $.
\end{rmk}

On putting $ z_1 = z $, $ z_2 = z^{-1} q^{-1} $ and $ z_3 =-1 $ in Theorem \ref{thm_main_new_recur}, we get the following corollary. Note that the left hand sides of both the corollary and the Jacobi triple product identity in  \eqref{eq_thm_Jacobi} are the same.

\begin{cor}
	\begin{align}
	\prod_{m=1}^{\infty} (1 + z q^{m}) (1 + z^{-1} q^{m-1}) (1 - q^{m}) = \sum_{m=0}^{\infty} 	\frac{q^{m} Q(m) }{(1-q)(1-q^2)(1-q^3) \cdots  (1-q^{m})},
	\end{align}
	where $ Q(m) := P(m)|_{z_1 = z, z_2 = z^{-1} q^{-1}, z_3 =-1}$.
\end{cor}

\subsection{A generalized $ k $-product identity}
Let $ S_n^k (z_1, z_2, \cdots z_k) $ be the elementary symmetric polynomial of degree $ n $ in variables $ z_1, z_2, \cdots, z_k $, i.e.,
\begin{equation}\label{eq_elem_symmetric_poly}
S_n^k (z_1, z_2, \cdots ,z_k):= \sum_{1 \leq i_1 < i_2 < \cdots < i_n \leq k } z_{i_1} z_{i_2} \cdots z_{i_n}.
\end{equation}

We recall that Theorem~\ref{thm_main_new_recur} was obtained by considering tilings of a $ 1 \times \infty $ board with the associated weight function given by
\eqref{eq_recur_weight}. If instead, we consider a  $ 1 \times \infty $ board $ F $ with the weight function
\begin{equation} \label{eq_def_weight_fun_k}
F_w(kn + j ) = z_j q^{n+1}, \qquad \text{for $ j = 1,2,3, \cdots ,k$, and $ n = 0,1,2,3, \cdots $},
\end{equation}	
then the generating function of all possible tilings will be
\begin{equation*}
H(F, z_1, z_2, \cdots ,z_k,q) = 	\prod_{m=1}^{\infty} (1 + z_1 q^m) (1 + z_2 q^m) \cdots  (1 + z_k q^m).
\end{equation*}
Following the same argument as in the proof of Theorem~\ref{thm_main_new_recur}, we can get the following recurrences
\begin{equation} \label{eq_recur_k_main_one}
F(j,m) = z_j q F(j+1,m-1) + F(j+1,m), \qquad \text{for $ 1 \leq j \leq k  $},
\end{equation}
and 
\begin{equation}\label{eq_recur_k_main_two}
F(k+1,m) = q^{m} F(1,m),
\end{equation}
from which we get the following general theorem, where we write $ F(m)  $ for $ F(1,m) $ for easing the notation.
\begin{thm}\label{thm_main_new_recur_k}
	Let $ z_1, z_2,\cdots, z_k $ and $ q $ be complex numbers with  $ |q| < 1 $. Let $ S_n^k (z_1, z_2, \cdots ,z_k) $ be as defined in \eqref{eq_elem_symmetric_poly}. 
	Let	\begin{align*}
		\prod_{m=1}^{\infty} (1 + z_1 q^m) (1 + z_2 q^m) \cdots  (1 + z_k q^m) &= \sum_{m=0}^{\infty} 	 F(m). 
	\end{align*}
	Then $ F(m) $ can be recursively defined as 	\begin{align} \label{eq_recur_pm_k}
	& F(m) = \frac{q^m}{1-q^m}  \sum_{n=1}^{k}  S_n^k (z_1, z_2, \cdots ,z_k) \, F(m-n), \\
	&\text{with } F(m) =0 \quad \text{for $ m <0 $}, \quad F(0) =1, \quad F(1) = \frac{q}{1-q} \sum_{i=1}^{k} z_i =  \frac{q}{1-q} S_1^k(z_1, z_2, \cdots ,z_k).\nonumber
	\end{align}
		
\end{thm}
\begin{proof}
	The proof of this theorem is similar to the proof of Theorem~\ref{thm_main_new_recur}. 
	First we show that  $ F(m) $, which is the generating function of all possible tilings of the board $ F $ with the associated weight function as defined in \eqref{eq_def_weight_fun_k},  satisfies \eqref{eq_recur_pm_k}. For obtaining \eqref{eq_recur_pm_k} we use the recurrence relations in \eqref{eq_recur_k_main_one} and \eqref{eq_recur_k_main_two}. We write these recurrences using operator $ D $, defined in Remark~\ref{rmk_operator}, and we obtain \eqref{eq_recur_pm_k} by the following calculation.
	\begin{align*}
	F(1,m) & = (z_1 qD + 1) (z_2 q D +1 ) \cdots (z_k q D + 1) F(k+1,m)  \\
	&= \left(\sum_{n=1}^{k} S_n^k(z_1, z_2, \cdots, z_k) q^n D^n \right) q^m F(1,m)  + q^m F(1,m),\\
\implies (1 -q^m) F(m)	&= q^m \sum_{n=1}^{k} S_n^k(z_1, z_2, \cdots, z_k) F(m-n).
	\end{align*}
	The result is now obvious.
\end{proof}
The proof of the following corollary is immediate from Theorem~\ref{thm_main_new_recur_k}.
\begin{cor}\label{cor_main_recur_k}
	Following the notation of Theorem~\ref{thm_main_new_recur_k}, for positive integers $ m $, define $ P(m) $ such that
\begin{equation*}
F(m)  =  \frac{q^{m} P(m) }{(q)_m}.       
\end{equation*}
Then,
\begin{align*}
\prod_{m=1}^{\infty} (1 + z_1 q^m) (1 + z_2 q^m) \cdots  (1 + z_k q^m) &= \sum_{m=0}^{\infty} 	\frac{q^{m} }{(q)_m} P(m). 
\end{align*}
Moreover, $ P(m) $ satisfies the following recurrence 
\begin{equation*}
P(m) = \sum_{i=1}^{k} \frac{q^{m-i} (q)_{m-1}}{(q)_{m-i}} \, S_i^k (z_1, z_2, \cdots, z_k) P(m-i).
\end{equation*}
\end{cor}

Finally, we give an alternate expression for  $ F(m) = F(1,m) $ in the following theorem.
\begin{thm} \label{thm_main_new_recur_k_alternate}
Let 
\begin{equation}\label{eq_new_F_m}
	F(m) = \sum_{m_{1} + m_{2} + \cdots + m_{k} =m} \, \frac{z_1^{m_{1}} q^{(m_{1}^2 + m_1)/2} \, z_2^{m_{2}} q^{(m_{2}^2 + m_2)/2}\, \cdots \cdots z_k^{m_{k}} q^{(m_{k}^2 + m_k)/2}  }{(q)_{m_{1}} (q)_{m_{2}}  \cdots (q)_{m_{k}} }, 
\end{equation}
where $ m_1, m_2, \cdots, m_k $ are non-negative integers. We also assume that $ F(0) =1 $, $ F(m) = 0$  for negative integers $ m $. Then for all positive integers $ m $ 
$ F(m) $ satisfies the following recurrence.
	\begin{align}
F(m) = \frac{q^m}{1-q^m}  \sum_{n=1}^{k}  S_n^k (z_1, z_2, \cdots, z_k) \, F(m-n). 
\end{align}
	
\end{thm}
\begin{proof}
We consider the $ 1 \times \infty $ board $ F $ with its associated weight function given by \eqref{eq_def_weight_fun_k}.
	 The board $ F $ can be viewed as $ k $ different $ 1 \times \infty $ boards, say $ F_j $, for $ j=1,2, \cdots , k $, such that the $ j $-th board $ F_j $ is the board consisting of all the positions of $ F $ which are congruent to $ j \mod k $. The weight at the $ i $-th position of the board $ F_j $ is $ z_j q^i  $. It can easily be checked that this is consistent with the definition of the weight function on $ F $ in \eqref{eq_def_weight_fun_k}. Clearly, the set of all tilings of the original board $ F $ with $ m $ black tiles is in bijection with the set of tilings of the new $ k $ boards, such that the board $ F_j $  is tiled with $ m_j $ black tiles and
	$ \sum_{j=1}^{k} m_j = m $.  The generating function for tiling $ F_j $  with $ m_j $ black tiles is
	 $ \frac{z_j^{m_{j}} q^{(m_{j}^2 + m_j)/2} }{(q)_{m_j}} $,
	 from 	\eqref{eq_euler_identity_m}. It is now obvious that
 \begin{align*}
 		\prod_{m=1}^{\infty} (1 + z_1 q^m) (1 + z_2 q^m) \cdots  (1 + z_k q^m) &= \sum_{m=0}^{\infty} 	 F(m),
 \end{align*}
such that
  \begin{align*}
 	F(m) = \sum_{m_{1} + m_{2} + \cdots + m_{k} =m} \, \frac{z_1^{m_{1}} q^{(m_{1}^2 + m_1)/2} \, z_2^{m_{2}} q^{(m_{2}^2 + m_2)/2}\, \cdots \cdots z_k^{m_{k}} q^{(m_{k}^2 + m_k)/2}  }{(q)_{m_{1}} (q)_{m_{2}}  \cdots (q)_{m_{k}} }.
 \end{align*}
Now the result follows from Theorem \ref{thm_main_new_recur_k}.
 \end{proof}


\end{document}